\documentclass[a4paper,oneside,12pt]{amsart}
\usepackage{typearea}
\usepackage{amssymb}
\usepackage{amsmath}
\usepackage{amsthm}
\usepackage{amsfonts,mathrsfs}
\usepackage{hyperref}

\usepackage{setspace}

\newcounter{thmcounter}

\makeatletter
\def\imod#1{\allowbreak\mkern10mu({\operator@font mod}\,\,#1)}
\makeatother

\newtheorem{lemma}[thmcounter]{Lemma}
\newtheorem{proposition}[thmcounter]{Proposition}
\newtheorem{corollary}[thmcounter]{Corollary}
\newtheorem{theorem}[thmcounter]{Theorem}

\newcommand{\scrB}{\mathscr{B}}
\newcommand{\IR}{{\mathbb R}}
\newcommand{\IN}{{\mathbb N}}
\newcommand{\IG}{{\mathbb G}}
\newcommand{\IGm}{{\mathbb G}_{\mathrm m}}
\newcommand{\IC}{{\mathbb C}}
\newcommand{\IRan}{{\mathbb R}_{\mathrm{an}}}

\newcommand{\IZ}{{\mathbb Z}}
\newcommand{\IF}{{\mathbb F}}

\newcommand{\IQ}{\mathbb{Q}}

\newcommand{\ssm}{\smallsetminus}

\newcommand{\mat}[2]{{\mathrm{Mat}}_{#1}({#2})}
\newcommand{\gl}[2]{{\mathrm{GL}}_{#1}({#2})}

\newcommand{\ma}[1]{\mathrm{m}({#1})}
\newcommand{\Ma}[1]{\mathrm{M}({#1})}

\renewcommand{\subset}{\subseteq}

\newcommand{\nbhd}[1]{{\mathcal{N}}({#1})}

\newcommand{\no}[2]{\mathrm{N}_{#2}({#1})}

\newcommand{\alg}[1]{{#1}^{\mathrm{alg}}}

\newcommand{\vol}[1]{\mathrm{vol}({#1})}

\newcommand{\ordS}{\mathrm{ord}}

\newcommand{\cX}{\mathcal{X}}
\newcommand{\cY}{\mathcal{Y}}
\newcommand{\cH}{\mathcal{H}}

\newcommand{\trop}[1]{\mathrm{Trop}({#1})}

\newcommand{\sv}{y}

\begin{document}

\title{The Norm of  Gaussian Periods}

\author{P. Habegger}
\address{Department of Mathematics and Computer Science, University of Basel, Spiegelgasse 1, 4051 Basel,
Switzerland}
\email{philipp.habegger\makeatletter @\makeatother unibas.ch}




\keywords{Gaussian period, Mahler measure, Diophantine Approximation,
  Ergodic Theory and Number Theory} 
\subjclass[2010]{Primary: 11R06, 
secondary: 05A15, 11J71, 11K60, 37A45}



\begin{abstract}
Gaussian periods are cyclotomic integers with a long history in
number theory and
connections to problems in combinatorics. 
We investigate the  asymptotic behavior of the absolute norm of a Gaussian
period and provide a rate of convergence in a 
case  of Myerson's Conjecture for 
periods of arbitrary odd length.
Our
method involves a result of Bombieri, Masser, and Zannier on unlikely
intersections in the algebraic torus as well as  work of the
author on
 the diophantine approximations to a set
definable in an o-minimal structure. In the appendix we make a result
of Lawton on Mahler measures quantitative.
\end{abstract}

\maketitle

\section{Introduction}

Let $f\ge 1$ be an  integer and $p$ a prime number. We are interested
 in the asymptotic behavior of the norm of
\begin{equation}
\label{eq:cyclosum}
  \zeta^{a_1}+\zeta^{a_1}+\cdots +\zeta^{a_f}
\quad\text{where}\quad
\zeta = e^{2\pi \sqrt{-1} /p}
\end{equation}
as  $a_1,\ldots,a_f\in\IZ $ and $p$ vary. 

These cyclotomic integers appear naturally in algebraic number theory. 
We identify the Galois
group of $\IQ(\zeta)/\IQ$ with $\IF_p^\times
=\IF_p\ssm\{0\}$.
Say $a_1,\ldots,a_f$ represent the elements of a subgroup
$G\subset\IF_p^\times$ of order $f$. Then the sum (\ref{eq:cyclosum}) is the
trace of $\zeta$ relative to the subfield of $\IQ(\zeta)$ fixed by
the said subgroup and it is
  called a Gaussian period.
It has
 degree $k = [\IF_p^\times :G] = (p-1)/f$ over $\IQ$, we refer to
 Chapter 10.10 of Berndt, Evans, and Williams' book \cite{BerndtEvansWilliams} for these
 and other facts.
A Gaussian period  
 can be expressed in terms of $p$ and $f$  for small values of $k$.
 Indeed,
 if $k=1$ then 
 $G=\IF_p^\times$ and the Gaussian period is of course $ \zeta + \zeta^2 +
 \cdots+\zeta^{p-1}= -1$. Gauss evaluated the sum if
$k=2$ and the minimal polynomial of a Gaussian period
has been computed if $k\le 4$.

The absolute norm of a Gaussian period appears 
 in combinatorial problems, cf.  Myerson's work
\cite{Myerson:CombI,Myerson:CombII}. 
Let $A_1,\ldots,A_k\in\IF_p^\times$
denote a complete set of representatives of $\IF_p^\times /G $.
Then the cardinality satisfies 
\begin{equation}
\label{eq:countN} 
  \# \{ (x_1,\ldots,x_k) \in G^k : A_1 x_1 + \cdots + A_k x_k = 0 \}
    - \frac 1p f^k = \frac{p-1}{p}\Delta 
\end{equation}
where 
\begin{equation*}
  \Delta=  \prod_{t \in \IF_p^\times/G}
\left(\sum_{g\in G} \zeta^{tg}\right);
\end{equation*}
 note that  
 $\zeta^t$ is well-defined  for $t\in\IF_p$ as is the sum.  
 If
$A_1 x_1+\cdots +A_k x_k$ were to attain all values of $\IF_p$
 equally often then $\Delta$ would vanish. 
As Myerson  \cite{Myerson:CombI} observed, this linear form attains
\textit{non-zero} values  equally often.
It is tempting to interpret  (\ref{eq:countN}) as an error term. 
But note that the 
 trivial estimate $\left|\sum_{g\in G}\zeta^{tg}\right|\le f$
leads to the upper bound  $(1-p^{-1})f^k$ for the modulus of
(\ref{eq:countN}). This bound exceeds 
$p^{-1}f^k$ for all $p\ge 3$. In this paper we will improve on the trivial bound 
if the length $f$ of the Gaussian period is a fixed prime and
$p$ is large. 

When well-defined, the logarithmic absolute norm of the Gaussian
period is
\begin{equation}
\label{eq:logmean}
\frac 1k \log|\Delta| =   \frac{1}{p-1}\sum_{t=1}^{p-1} \log\left|   
\sum_{g\in G} \zeta^{tg}
\right|.
\end{equation}
Our Theorem \ref{thm:sprime} below determines the asymptotic behavior
of this value as
$p\rightarrow\infty$ when $f=\# G$  is a fixed prime.
Before stating our first result, we survey
 what is known for groups $G$ of order $f$. 

Certainly (\ref{eq:logmean}) vanishes if $G$ is trivial. 
If $f =2$, then $G=\{\pm 1\}$ and $p$ is odd. 
Note that $\zeta + \zeta^{-1} = \zeta^{-1}(\zeta^2+1)$ is a unit in
$\IQ(\zeta)$. So (\ref{eq:logmean}) is again zero. The value of
$\Delta$, i.e. its sign, can be computed using the Kronecker symbol.

Already the case $f=3$ is more involved. It requires the logarithmic Mahler measure 
\begin{equation*}
  \ma{P} = \int_{0}^1 \cdots \int_0^1 \log \left| P\left(e^{2\pi
    \sqrt{-1}x_1},\ldots,e^{2\pi \sqrt{-1}x_n}\right)\right|
dx_1\cdots dx_n,
\end{equation*}
 of a non-zero Laurent polynomial $P\in
\IC[X_1^{\pm 1},\ldots,X_n^{\pm 1}]$; for the fact that this integral converges and
other properties we refer to Chapter 3.4 in Schinzel's book \cite{Schinzel}.
If
 $p\equiv 1 \imod 3$, then  $\IF_p^\times$ contains an element
$\theta$ of order $3$.  Myerson, cf. Lemma 21 \cite{Myerson:CombII},  proved
\begin{equation}
\label{eq:myersonlimit}
  \frac{1}{p-1} \sum_{t=1}^{p-1} \log\left|\zeta^t +
  \zeta^{t \theta}+\zeta^{t \theta^2}\right| = \ma{1+X_1+X_2}+o(1)
\end{equation}
as $p\rightarrow \infty$.
  The logarithmic Mahler measure of $1+X_1+X_2$
 was evaluated by Smyth \cite{SmythMahler}  and equals
 $L'(-1,\chi)$ where $\chi$ is the non-trivial character modulo $3$
and $L(\cdot,\chi)$ is its associated Dirichlet $L$-function. 
Duke \cite{Duke:Combinatorial} gave a new proof of (\ref{eq:myersonlimit})
which extended to a larger class of vectors in $\IF_p^3$ containing
the exponent vector $(1,\theta,\theta^2)$. Moreover, he  provided a rate of convergence.

Due to a fortunate factorization, the case $f=4$ is similar to  order
$2$.  
Indeed, 
Theorem 6 \cite{Myerson:CombII} implies 
$\Delta = \pm 1$ if $p\equiv 1 \imod 4$ and also determines the sign. 
So the limit in question is again zero.

For higher order, the approaches in \cite{Myerson:CombII} and
\cite{Duke:Combinatorial} break down. 
 But Myerson's Conjecture \cite{Myerson:CombI} 
predicts convergence of (\ref{eq:logmean}) as $p\rightarrow \infty$ 
and  the limit point.
The full conjecture is more general as it also covers subgroups of
$\IF_{q}^\times$ where $q$ is a fixed power of $p$.

Myerson's Conjecture has an ergodic flavor. Indeed,
using methods from ergodic theory,
 Lind, Schmidt, and Verbitskiy 
\cite{LindSchmidtVerbitskiy} proved  convergence in the following
setting.
They suitably
 averaged the value of the logarithm of the modulus of a polynomial
evaluated at a finite subgroup of roots of unity. 
Their polynomials are required to satisfy an intersection theoretical 
property with respect to the maximal compact subgroup of
$(\IC\ssm\{0\})^n$. 
They  computed the limit  of this average for certain sequences 
of groups of roots of
unity.

In this paper we concentrate on the special case when $G$ has fixed odd
order. We   prove that (\ref{eq:logmean}) converges and compute the
limit.

Our method is based on a recent result  \cite{approxdefinable}  of the author on diophantine
approximation of  sets definable in a polynomially bounded o-minimal
structure. It counts strong rational approximations to a definable
set and is related to the Pila-Wilkie Counting Theorem
\cite{PilaWilkie}. 
Our approach is quantitative in the sense that it
can provide a rate of convergence.

The following theorem is  the special case of our main
result when $\#G$  is an
odd prime.
We refer to Corollary \ref{cor:myerson} below for a more general
statement.
This corollary contains the case  $q=p$ of Myerson's
Conjecture.

\begin{theorem}
  \label{thm:sprime}
Suppose $f$ is an odd prime. For a prime  $p$  with $p\equiv 1
\imod {f}$ let  $G_p\subset\IF_p^\times$ denote the subgroup of order $f$. Then
  \begin{equation}
\label{def:mahlermeasurethm1}
    \frac{1}{p-1} \sum_{t=1}^{p-1} \log\left| \sum_{g\in G_p} e^{2\pi\sqrt{-1}\frac{tg}{p}}
    \right| = \ma{1+X_1+\cdots + X_{f-1}} + O\left(p^{-\frac{1}{5(f-1)^2}}\right)
  \end{equation} 
as $p\rightarrow\infty$; in particular, the 
logarithm is well-defined for all sufficiently large  $p$. 
\end{theorem}

We have the estimate $\ma{1+X_1+\cdots +X_{f-1}} \le \frac 12 \log f $
by Corollary 6 in Chapter 3.4 \cite{Schinzel}. This justifies  the
treatment of
(\ref{eq:countN}) as an error term if $f=\#G$ is a prime and $p$ is
large.

The value $\ma{1+X_1+\cdots + X_{f-1}}$ is non-zero if $f\ge
3$. This implies 
an amusing  corollary of Theorem \ref{thm:sprime} on Gaussian periods
that are units. 

\begin{corollary}
\label{cor:finitelymanyunits}
  Suppose $f$ is an odd prime. There are at most finitely many primes $p$
  with $p\equiv 1 \imod{f}$ such that
  \begin{equation*}
    \sum_{g\in G} e^{2\pi\sqrt{-1}\frac gp}
  \end{equation*}
is an algebraic unit where $G\subset\IF_p^\times$ denotes the subgroup
of order $f$. 
\end{corollary}

   Gaussian periods and their generalizations were
investigated by Duke, Garcia, and Lutz  \cite{DukeGarciaLutz} from
several points of view. In Theorem 6.3 they prove  that the
Galois orbit of a Gaussian period becomes equidistributed in a suitable
sense. 
Our   average (\ref{eq:logmean}) involves the logarithm whose
singularity at the origin often makes applying equidistribution
directly impossible, see for example Autissier's example \cite{Autissier:counterexample}.

Our main technical result is Theorem \ref{thm:mainapp} below. It
essentially amounts to a convergence result
when  averaging over   groups of roots of unity of
prime order. It is used to deduce the theorem and corollary above.

 The main
difficult when $f=\# G \ge 5$  is that the integrand 
\begin{equation*} 
  1+X_1 + \cdots  + X_{f-1}
\end{equation*}
 in the logarithmic Mahler measure (\ref{def:mahlermeasurethm1}) has  singularities along a
 positive dimensional real semi-algebraic set. 

Our approach requires new tools and we now give  a brief
 overview of the proof of Theorem \ref{thm:sprime}. 
As in Duke's work \cite{Duke:Combinatorial} we start with a
 basic observation; to simplify notation we set $n=f-1$. Let  $a=(a_1,\ldots,a_n)\in\IZ^n$,
 then 
 \begin{equation*}
   \prod_{t=1}^{p-1}\left(   1+\zeta^{t a_1}+\cdots +\zeta^{t a_n}\right)
 \end{equation*}
is the product of
\begin{equation*}
  P_a=1+X^{a_1}+\cdots + X^{a_n}
\end{equation*}
  evaluated at all roots
of unity of order $p$. If $a_n> a_{n-1} > \cdots > a_1 > 0$, then $P_a$
is a monic polynomial and the product above is
 the resultant of $P_a$ and 
$1+X+\cdots +X^{p-1}$. If this resultant is non-zero, 
then by symmetry properties  we
find that (\ref{eq:logmean}) equals
\begin{equation}
\label{eq:altnorm}
-\frac{\log(n+1)}{p-1} + \frac{1}{p-1}\sum_{i=1}^d \log|\alpha_i^p-1|.
\end{equation}
where $\alpha_1,\ldots,\alpha_d$ are the roots of $P_a$.
Comparing $\frac{1}{p-1} \log |\alpha_i^p-1|$ to the
local contribution  $\log
\max\{1,|\alpha_i|\}$ of the logarithmic Mahler measure $\ma{P_a}$
is a crucial aspect of the problem at hand;
see Section \ref{sec:notation} for details on the Mahler measure.
Indeed, by a 
result of Lawton \cite{Lawton} the value  $\ma{P_a}$ converges towards the
logarithmic Mahler measure of a multivariate  polynomial as in Theorem \ref{thm:sprime} if $|a|\rightarrow \infty$
for $a$ in sufficiently
general position.
In a self-contained appendix, we provide a quantitative version of
Lawton's Theorem, see Theorem \ref{thm:lawtonquant}.

 Baker's theory on linear
forms in logarithms yields a lower bound for non-zero values of
$|\alpha_i^p-1|$. But the current estimates  are not strong enough to
directly establish Theorem   
\ref{thm:sprime}, see Duke's comment after the proof of this Theorem 3
\cite{Duke:Combinatorial}. 
However, as we shall see below,  strong lower bounds for
$|\alpha_i^p-1|$ are
available if $|\alpha_i|\not=1$. Indeed, 
 $|\alpha_i^p-1|\ge ||\alpha_i|^p-1|\ge ||\alpha_i|-1|$. 
If $|\alpha_i|\not=1$ we will 
use an old result of
Mahler on the separation of distinct roots of a polynomials to bound
$\left||\alpha_i|-1\right|$ from below. 
If $\alpha_i$
lies in $\mu_\infty$, the set of all roots of unity in $\IC$, then
a sufficiently strong lower bound for $|\alpha_i^p-1|$ follows from
simpler considerations.

This estimate gives  us sufficient control  on each term in the sum
(\ref{eq:altnorm}) subject to the condition that 
$P_a$ does not have any root in $S^1\ssm\mu_\infty$, 
here $S^1$ is the unit circle in $\IC$. 
But it seems unreasonable to expect this hypothesis to hold for
all $a$. 
To address this concern we use  symmetry in
(\ref{eq:logmean}). Indeed, this mean is invariant under translating
$a$ by an element of $p\IZ^n$ and also by replacing $a$ by $ta$ with
$t\in\IZ$ coprime to $p$. We exploit this symmetry by  using
 a  result of Bombieri, Masser, and Zannier
\cite{BMZGeometric} on unlikely intersections. This in combination
with Dirichlet's Theorem in diophantine approximation 
 allows us to assume that
$P_a$ has no roots on $S^1\ssm\mu_\infty$ after a suitable transformation as
described above. 
Here the parity assumption on $f$ in Corollary
\ref{cor:myerson} below is used.

At this point we have a sufficient lower bound for each term of
the sum (\ref{eq:altnorm}). However, the method cannot proceed if too many terms are close to
this bound.  Duke \cite{Duke:Combinatorial}
 already use the following basic principle. Suppose
that for some $\alpha=\alpha_i$  the distance $|\alpha^p-1|$ is small, i.e.
at most a fixed power of $p^{-1}$. Than one can expect that
$\alpha$ is close to some $\zeta\in\mu_p$ where $\mu_p$ is the set of
roots of unity in $\IC$ of order dividing $p$. As
$P_a(\alpha)=0$ we find that 
\begin{equation}
\label{eq:smallsum}
\left|  1  + \zeta^{a_1} + \cdots + \zeta^{a_n} \right|
\end{equation}
is small. 
Myerson \cite{Myerson:sumofrootsof1} proved a lower bound for non-vanishing sums of roots
of unity if $n=1,2,$ and $3$. His estimates are
are polynomial in $p^{-1}$ and
 are strong enough  to imply Duke's result.
However, for fixed $n\ge 4$ only exponential bounds such as 
$(n+1)^{-p}$ are known to hold in general. They  are not good enough for our purposes. 

If $|\alpha_i^p-1|$ is small for many $i$, we are able to show that
(\ref{eq:smallsum}) is small for many $p$-th roots of unity
$\zeta$. This situation can be analyzed using 
the following theorem that counts small sums of roots of unity of
prime order.
Its proof requires recent work of the author \cite{approxdefinable}
on diophantine approximations on  definable sets in
an o-minimal structure.

\begin{theorem}
\label{thm:countrootsof1}
Let $n\ge 1$. For all $\epsilon > 0$ there exist constants $c=c(n,\epsilon)\ge 1$ and 
$\lambda = \lambda(n,\epsilon)\ge 1$ with the following property. 
If $p$ is a prime and $\zeta_1,\ldots,\zeta_n\in\mu_p$ satisfy
 $1+\zeta_1 + \cdots +
\zeta_n\not=0$, then 
\begin{equation}
\label{eq:count}
  \#\left \{ t\in \IF_p : |1+\zeta_1^t+ \cdots
  +\zeta_n^t|<c^{-1}p^{-\lambda} \right\} \le cp^\epsilon. 
\end{equation}
\end{theorem}

As this paper was being finished up, Dimitrov \cite{Dimitrov}
announced a extension to more general polynomials
of Lind, Schmidt, and Verbitskiy's work  for  subgroups
that are Cartesian powers.
His approach used ideas from diophantine
approximation and is independent from ours.

We hope to expand the connection between counting points approximating
a definable set and questions related to ergodic theory in future
work. 

\section{Notation}
\label{sec:notation}

The supremum norm on $\IR^n$ is $|\cdot|$ for any $n\ge 1$. 
We have already seen the definition of the  logarithmic Mahler measure $\ma{P}$
if $P\in \IC[X_1^{\pm 1},\ldots,X_n^{\pm 1}]\ssm\{0\}$. The Mahler
measure of $P$ is
$\Ma{P} = e^{\ma{P}}$. 

The absolute logarithmic Weil height, or just height, of  an algebraic
number $\alpha$  with minimal polynomial $P$ in $\IZ[X]$ and
leading term $p_0\ge 1$ is
\begin{equation*}
  h(\alpha) = \frac{1}{[\IQ(\alpha):\IQ]} \ma{P}
 = \frac{1}{[\IQ(\alpha):\IQ]} \log\left(p_0 \prod_{\substack{z\in \IC \\ P(z)=0}}
\max\{1,|z|\} \right) 
\end{equation*}
where the second equality follows from Jensen's Formula. 
We write $H(\alpha)= e^{h(\alpha)}$. Moreover,  we set
$H(\alpha_1,\ldots,\alpha_n) = \max\{
H(\alpha_1),\ldots,H(\alpha_n)\}$ if
$\alpha_1,\ldots,\alpha_n$ are algebraic.

\section{Algebraic Numbers Close to the Unit Circle}

An algebraic number $\alpha\in \IC\ssm\{1\}$ of degree 
$D=[\IQ(\alpha):\IQ]$ can be bounded away from $1$ using
Liouville's Inequality, Theorem 1.5.21 \cite{BG}, 
\begin{equation*}
\log  |\alpha-1| 
\ge -D\log 2 - D h(\alpha).
\end{equation*}
The modulus $|\alpha| = \sqrt{\alpha\overline\alpha}$ is again an
algebraic number, here and below $\overline\cdot$ denotes complex conjugation. Its height satisfies
\begin{equation*}
  h(|\alpha|) \le \frac 12 h(\alpha\overline\alpha)\le  h(\alpha)
\end{equation*}
since $h(\overline\alpha)=h(\alpha)$. 
If $\alpha$ is real, then clearly $\IQ(|\alpha|)=\IQ(\alpha)$. 
However, for $D\ge 2$  we only have
$[\IQ(|\alpha|):\IQ]\le D(D-1)$ and equality is possible. So Liouville's Inequality 
applied to $|\alpha|$ gives
\begin{equation*}
\log  \left| |\alpha|-1 \right| \ge -D(D-1)\log 2 - D(D-1)h (\alpha)
\end{equation*}
if $|\alpha|\not=1$. 
We will use a result of Mahler to improve on the
dependency in $D$ in front of $\log 2$.

\begin{theorem}[Mahler]
Let $P\in\IZ[X]$ be a polynomial with $D=\deg P \ge 2$. If
$z,z'\in\IC$ are distinct roots of $P$, then 
\begin{equation*}
|z'-z| > \sqrt 3 D^{-(D+2)/2} \Ma{P}^{-(D-1)}.
\end{equation*}
\end{theorem}
\begin{proof}
  We may assume that $P$ has no multiple roots over $\IC$ after
replacing it by its
squarefree part.
The estimate then  follows from  Theorem 2 \cite{Mahler:DiscIneq} as
the absolute value of the new discriminant is at least $1$. 
\end{proof}

\begin{lemma}
\label{lem:alphapminus1}
  Let $\alpha\in \IC$ be an algebraic number of degree
  $D=[\IQ(\alpha):\IQ]$. If $|\alpha|\not=1$  then 
  \begin{equation*}
\log |\alpha^p-1|\ge \log ||\alpha|-1|\ge -1-(D+1)\log(2D)
-2(2D-1)D h(\alpha)
  \end{equation*}
for all integers $p\ge 1$. 
\end{lemma}
\begin{proof}
The first inequality  follows
from $|\alpha^p-1|\ge ||\alpha|^p-1| = ||\alpha|-1|
\cdot ||\alpha|^{p-1}+\cdots + 1|\ge ||\alpha|-1|$. 
To prove the second inequality 
we may assume $|\alpha|\ge 1/2$, in particular $\alpha\not=0$.  
  Let $P\in\IZ[X]$ denote the minimal polynomial of $\alpha$. We will
 apply
  Mahler's Theorem to  $F=P(X) P(1/X) X^D \in\IZ[X]$. 
Observe that $F(\alpha)=F(\overline\alpha^{-1})=0$ and $\deg F = 2D$.
Therefore, $|\alpha-\overline \alpha^{-1}| > \sqrt 3
(2D)^{-(2D+2)/2} \Ma{F}^{-(2D-1)}$ since $|\alpha|\not=1$. As $\Ma{P(1/X)X^D}=\Ma{P}$ and
since the Mahler measure is multiplicative, we find, after multiplying
with $|\alpha|=|\overline\alpha|$, that   
\begin{equation*}
\left ||\alpha|^{2}-1\right| > \sqrt 3 |\alpha| (2D)^{-(D+1)}
  \Ma{P}^{-2(2D-1)}. 
\end{equation*}
Observe that $\log \Ma{P} = D h(\alpha)$ and $||\alpha|-1| = ||\alpha|^2-1| / ( |\alpha|+1)$. 
Therefore,
\begin{equation*}
  \left||\alpha|-1\right| >\sqrt 3 \frac{|\alpha|}{|\alpha|+1} 
(2D)^{-(D+1)} \Ma{P}^{-2(2D-1)}
 \ge \frac {\sqrt 3}{3}
 (2D)^{-(D+1)}
  e^{-2(2D-1)Dh(\alpha)}
\end{equation*}
using $|\alpha|\ge 1/2$. We conclude the proof by taking the
logarithm. 
\end{proof}

\section{A First Estimate}

Let $n\ge 1$ and $p\ge 2$ be  integers.
For $a=(a_1,\ldots,a_n)\in\IZ^n$ we define
\begin{equation}
\label{eq:defDelta}
  \Delta_p(a) = \prod_{t=1}^{p-1}
|1+\zeta^{t a_1}+\cdots +\zeta^{t a_n}|
\quad\text{where}\quad \zeta = e^{2\pi \sqrt{-1}/p}.
\end{equation}
If $p$ is  a prime, then $\Delta_p(a)$ is the 
$\IQ(\zeta)/\IQ$ norm of the cyclotomic integer
$1+\zeta^{a_1}+\cdots + \zeta^{a_n}$ up-to sign. 
We 
attach to $a$  the lacunary Laurent polynomial
\begin{equation}
\label{def:Pa}
  P_a = 1+X^{a_1} + \cdots + X^{a_n}\in \IZ[X^{\pm 1}].
\end{equation}
Say $e = \max\{0,-a_1,\ldots,-a_n\}\ge 0$, then
\begin{equation}
\label{eq:Ppoly}
  X^{e} P_a  = X^e +X^{e+a_1}+\cdots + X^{e+a_n}
\end{equation}
is a polynomial with integral coefficients,  non-zero constant term,
and  degree $d=\max\{a_i -a_j : 0\le i,j\le n\}\le 2|a|$, where $a_0=0$. 
We write $\alpha_1,\ldots,\alpha_d\in\IC$ for the zeros of $X^e P_a$ 
with multiplicity, i.e. $X^eP_a = p_0(X-\alpha_1)\cdots (X-\alpha_d)$
where $p_0\ge 1$ is the leading term of $P_a$. We note that $\alpha_i
\not = 0,1$ for all $i$. 

 Our goal in this section is to bound
\begin{equation}
\label{eq:goal}
\left|  \frac{1}{p-1} \log \Delta_p(a) - \ma{P_a} \right|
\end{equation}
from above, here
$\ma{P_a}=\ma{X^e P_a} = \log p_0+\sum_{i=1}^d \log
\max\{1,|\alpha_i|\}$ is the logarithmic Mahler measure of $P_a$. 

As $\Delta_p(a)$ is essentially a resultant we can rewrite it
as a product over the roots $\alpha_i$.
This will allow us to 
express the difference (\ref{eq:goal}) in terms of these roots. 

In the next 4 lemmas  we obtain several statements on the roots
$\alpha_i$ in terms of $a\in\IZ^n$.

\begin{lemma}
\label{lem:altDelta}
We have $\Delta_p(a) = p_0^{p} (n+1)^{-1} \prod_{i=1}^d |\alpha_i^p-1|$. 
\end{lemma}
\begin{proof}
A variant of this calculation can also be found in the proof of
Duke's Theorem 3 \cite{Duke:Combinatorial}. 
We have
\begin{alignat*}1
  \Delta_p(a) &= \prod_{t=1}^{p-1} |P_a(\zeta^t)| 
=p_0^{p-1}\prod_{i=1}^d \left(\prod_{t=1}^{p-1} |\zeta^t-\alpha_i|\right)
= p_0^{p-1} \prod_{i=1}^d \left|\frac{1-\alpha_i^p}{1-\alpha_i}\right|.
\end{alignat*}
The lemma follows since $p_0\prod_{i=1}^d(1-\alpha_i) = P_a(1) = n+1$. 
\end{proof}

Each $\alpha_i$ is an algebraic number with $D_i = [\IQ(\alpha_i):\IQ]\le
d$ whose height is bounded by the next lemma. 

\begin{lemma}
\label{lem:heightbound}
 Let $i \in \{1,\ldots,d\}$, then $D_i
h(\alpha_i)\le \ma{P_a}\le 
  \log(n+1)$. 
\end{lemma}
\begin{proof}
Recall that $\alpha_i$ is a root of the polynomial in
(\ref{eq:Ppoly}) which,  
by the Gauss Lemma, is divisible by the minimal polynomial $Q$  in $\IZ[X]$ of $\alpha_i$.
 So $D_i h(\alpha_i) = \ma{Q}\le \ma{X^eP_a}=\ma{P_a}$ as the logarithmic Mahler measure is
additive and non-negative on $\IZ[X]\ssm\{0\}$. 
By 
Corollary 6, Chapter 3.4 \cite{Schinzel}
the Mahler measure $\Ma{P_a}$ is at most the euclidean norm of the 
coefficient vector of $P_a$. This gives $\ma{P_a}\le \log(n+1)$.
\end{proof}

We now come to a lower bound for $|\alpha_i^p-1|$ which is independent
of $p$ under the assumption that $\alpha_i$ lies off the unit
circle or is a root of unit of order not  divisible by $p$. 

\begin{lemma}
\label{lem:loglb1}
Suppose $a\not=0$, let  $i\in
\{1,\ldots,d\}$, and let $p\ge 1$ be an integer. 
\begin{enumerate}
\item [(i)]
If $|\alpha_i|\not=1$, then 
$\log |\alpha_i^p-1|\ge - 18\log(n+1) |a|\log(2|a|)$. 
\item[(ii)] If $\alpha_i$  is a root of unity and $\alpha_i^p\not=1$, then 
$\log |\alpha_i^p-1|\ge -2\log(2|a|)$. 
\end{enumerate}
\end{lemma}
\begin{proof}
Observe that $|\alpha_i|\not=1$ implies $n\ge 2$. 
According to Lemmas \ref{lem:alphapminus1} and
\ref{lem:heightbound} we have 
\begin{alignat*}1
\log|\alpha_i^p-1|&\ge -1 -(D_i+1)\log(2D_i) -2(2D_i-1)\log(n+1) \\
&\ge -(D_i+1)\log(2D_i) -4D_i \log(n+1).
\end{alignat*}
 Now $D_i\le
2|a|$ by (\ref{eq:Ppoly}). The first part of the lemma follows from 
\begin{equation*}
 (2|a|+1)\log(4|a|) +8|a|\log(n+1) \le 18 \log(n+1) |a|\log (2|a|). 
\end{equation*}

The second part is more elementary. Let $m\ge 2$ be the multiplicative
order of $\alpha_i^p$.
If $m\ge 3$, then  $|\alpha_i^p-1|\ge \sin(2\pi/m)\ge 2/m$.
The bound $|\alpha_i^p-1|\ge 2/m$ certainly also holds
 for $m=2$.
It is well-known that Euler's totient function $\varphi$ satisfies 
 $\varphi(m)\ge \sqrt{m/2}$. As
 $\varphi(m)=[\IQ(\alpha_i^p):\IQ]\le D_i$ we find 
$m\le 2D_i^2$. Hence $\log|\alpha_i^p-1|\ge -2\log D_i\ge
-2\log(2|a|)$, as $D_i\le 2|a|$. 
\end{proof}

This last lemma allows us to compare $\log|\alpha_i^p-1|$ with the
corresponding contribution $p\log\max\{1,|\alpha_i|\}$ in the
logarithmic Mahler
measure. 
\begin{lemma}
\label{lem:localcomparison}
Suppose $a\not=0$, let $i\in
\{1,\ldots,d\}$, and let $p\ge 1$ be an integer. 
\begin{enumerate}
\item [(i)]
  If $|\alpha_i|\not=1$,  then 
  \begin{equation}
\label{eq:localcomparison1}
    \bigl|\log|\alpha_i^p-1| - p \log \max\{1,|\alpha_i|\} \bigr|\le 
18 \log(n+1) |a|\log(2|a|).
  \end{equation}
\item[(ii)] If $\alpha_i$ is a root of unity and
  $\alpha_i^p\not=1$, then 
  \begin{equation*}
    \bigl|\log|\alpha_i^p-1|\bigr|\le 2\log(2|a|). 
  \end{equation*}
\end{enumerate}
\end{lemma}
\begin{proof}
For the proof of part (i) let us first assume $|\alpha_i|<1$. Then $|\alpha_i^p-1|\le 2$ and
Lemma \ref{lem:loglb1}(i) yields
$\left|\log |\alpha_i^p-1|\right|\le 18 \log(n+1) |a|\log(2|a|)$, as
desired. 

If $|\alpha_i|>1$ we use that $\alpha_i^{-1}$ is a root of $P_{-a}$. 
We obtain the same bound as before for
  $\bigl|\log |\alpha_i^{-p}-1|\bigr| = 
  \bigl| \log |\alpha_i^{p}-1| - p\log|\alpha_i|\bigr|$ and this completes the
  proof of (i). 

To prove (ii) we argue as in the case $|\alpha_i|<1$ above
but use Lemma \ref{lem:loglb1}(ii).
\end{proof}

Suppose for the  moment that  all $\alpha_i$ satisfy
$|\alpha_i|\not=1$ and  for sake of simplicity also $p_0=1$. By  Lemma \ref{lem:altDelta} 
the bound given in part (i) of the last lemma
leads to the bound
\begin{equation}
\label{eq:zerothbound}
 \frac{d|a|\log(2|a|)}{p}\le  \frac{2|a|^2\log(2|a|)}{p}
\end{equation}
 for (\ref{eq:goal})  up-to
 a factor depending only on $n$. However, this estimate is not strong 
enough for our aims due to the contribution $|a|^2/p$.

To remedy this we begin by splitting up the roots $\alpha_i$ into
two parts depending on a parameter $\lambda \ge 1$. The first part 
 \begin{equation}
\label{def:scrB}
  \scrB = \scrB(p,a,\lambda) = \left\{ i : |\alpha_i|<1\text{ and }
|\alpha_i^{p}-1|< p^{-\lambda}\right\}
\cup \left\{ i : |\alpha_i|>1\text{ and
  }|\alpha_i^{-p}-1|< p^{-\lambda}\right\}
    \end{equation}
corresponds to those roots  whose $p$-th power is 
excessively 
 close to $1$.
The second part is the complement
\begin{equation*}
  \{1,\ldots,d\} \ssm \scrB. 
\end{equation*}
Later we will bound the cardinality of $\scrB$. 

\begin{proposition}
\label{prop:firstbound}
Let $\lambda\ge 1$,  let $a\in\IZ^n\ssm\{0\}$, and let $p\ge 1$ be an
integer  satisfying $|a|\le  p$. 
Suppose  that the only roots of $P_a$ that lie on the unit circle are
roots of unity of order not dividing $p$. Then $\Delta_p(a)\not=0$ and
  \begin{equation}
\label{eq:firstbound}
\frac{1}{p-1} \log \Delta_p(a) = \ma{P_a}  + O\left( 
  \frac{|a|\log(2p)}{p}(\lambda + \#\scrB)\right)
  \end{equation}
where the implied constant depends only on $n$.  
\end{proposition}
\begin{proof}
We recall Lemma \ref{lem:altDelta}, it implies 
$\Delta_p(a)\not=0$ under the given circumstances. 
Let $\alpha_1,\ldots,\alpha_d$ be the roots of $P_a$ with
multiplicities, as
  above.

Say first $i\in \{1,\ldots,d\}\ssm \scrB$. 
  If $|\alpha_i|<1$, then $|\alpha_i^p-1|\ge p^{-\lambda}$. Thus
$\bigl|\log |\alpha_i^p-1|\bigr| \le \lambda \log(2p)$  since $\lambda\log(2p)\ge \log
  2$. 
For $|\alpha_i|>1$ we proceed similarly and obtain 
$\bigl|\log |\alpha_i^p-1|-p\log|\alpha_i|\bigr|\le \lambda\log(2p)$. 
If $|\alpha_i|=1$, then by hypothesis $\alpha_i$ is a root of unity
whose order does not divide $p$, so
 Lemma \ref{lem:localcomparison}(ii) implies
 $\left|\log |\alpha_i^p-1|\right| \le 2\log(2|a|)\le 2\log(2p)$.

We recall $d\le 2|a|$ and sum  over $\{1,\ldots,d\}\ssm \scrB$ to find
\begin{alignat}1
\label{eq:scrBcompl}
\sum_{\substack{i=1 \\ i\not\in \scrB}}^d
  \bigl|\log|\alpha_i^p-1| - p \log \max\{1,|\alpha_i|\} \bigr|
&\le (d-\#\scrB)\max\{\lambda \log(2p),2\log(2p)\}  \\ &\le 4|a|
  \log (2p)\lambda. \nonumber 
\end{alignat}

If $i\in \scrB$ then $|\alpha_i|\not=1$. We apply 
 Lemma
\ref{lem:localcomparison}(i) and use $|a|\le p$. The bound (\ref{eq:localcomparison1}) holds for
$\#\scrB$ roots and therefore
\begin{alignat*}1
\sum_{i\in \scrB}
  \bigl|\log|\alpha_i^p-1| - p \log \max\{1,|\alpha_i|\} \bigr|
&\le 18 \log(n+1) |a|\log(2p)\#\scrB. 
\end{alignat*}
We combine this bound with (\ref{eq:scrBcompl}) to obtain
\begin{equation}
\label{eq:scrBandcompl}
  \sum_{i=1}^d 
  \bigl|\log|\alpha_i^p-1| - p \log \max\{1,|\alpha_i|\} \bigr|
\le 18\log(n+1)|a|\log(2p)(\lambda + \#\scrB). 
\end{equation}

The logarithmic Mahler measure of $P_a$ is $ \log p_0 + \sum_{i=1}^d
\log\max\{1,|\alpha_i|\}$ where $p_0\ge 1$ is the leading term of $P_a$. 
We use 
  Lemma \ref{lem:altDelta} and apply the triangle inequality
to find that $|\frac{1}{p-1} \log\Delta_p(a)-\ma{P_a}|$ is at most 
\begin{alignat*}1
\frac{1}{p-1} &\left|\log\left(\frac{p_0^{p}}{n+1}\right) -(p-1)\log p_0 \right| + \frac{1}{p-1} \sum_{i=1}^d 
\left|\log |\alpha_i^p-1| - (p-1)\log \max\{1,|\alpha_i|\}\right| \\
& \le \frac{\log(n+1)}{p-1}
+  18\log(n+1)
\frac{|a|\log(2p)}{p-1}(\lambda + \#\scrB) 
+\frac{1}{p-1}\ma{P_a}
\end{alignat*}
where we used  (\ref{eq:scrBandcompl}). 
From Lemma \ref{lem:heightbound} we deduce $\ma{P_a}\le \log(n+1)$. 
So
 (\ref{eq:firstbound}) 
holds true. 
\end{proof}

If we ignore for the moment all logarithmic contributions, then we
have traded in $d$ in (\ref{eq:zerothbound}) for $\lambda + \#\scrB$ in 
(\ref{eq:firstbound}).

Before continuing we  make  two elementary, but important,
observations on  symmetry properties of $\Delta_p(a)$ when $p$ is a
prime. 

\begin{lemma}
\label{lem:deltasymmetry}
Let $p$ be a prime  and let $a\in\IZ^n$ be arbitrary.
  \begin{enumerate}
  \item [(i)] If $a'\in\IZ^n$ with $a\equiv a'\imod p$, then 
$\Delta_p(a)=\Delta_p(a')$. 
\item[(ii)] If $t\in\IZ$ with $p\nmid t$, then $\Delta_p(a) =
  \Delta_p(ta)$. 
  \end{enumerate}
\end{lemma}
\begin{proof}
  Part (i) follows using the
 definition (\ref{eq:defDelta}) and  $\zeta^p = 1$. For part (ii) observe that
  $\zeta\mapsto \zeta^t$ permutes the factors in (\ref{eq:defDelta}) since $p\nmid
  t$. 
\end{proof}

Using this lemma we will transform $a'= ta-b$ with $p\nmid
t$ and $b\in p\IZ^n$ such that $|a'|$ is small compared to $p$. 
This will be done using 
Dirichlet's Theorem from diophantine approximation. 
A theorem of
Bombieri-Masser-Zannier  leads to
a criterion that rules out  that $P_{a'}$ has roots on the unit circle
of infinite order.
This opens the door to applying the previous proposition. 

In a final step we will need a strong upper bound for $\scrB$ for a
sufficiently large but fixed $\lambda$. This is where  counting
rational points close to a definable sets comes into play.

\section{Lacunary Polynomials with Roots off the Unit Circle}
\label{sec:lacunary}

Say $n\ge 2$. In this section we investigate
a condition on $a\in\IZ^n$ such that $P_a$, as defined in (\ref{def:Pa}), does not
 vanish at any point 
of $S^1\ssm \mu_\infty$. 
It will prove useful to restrict $a$ to a subgroup
$\Omega\subset\IZ^n$
 and we will
introduce a condition on $\Omega$ that ensures that $P_a$ does not
have any roots in $S^1\ssm\mu_\infty$ for  $a\in\Omega$ in
 general position.

Our condition is based on a theorem of Bombieri, Masser, and
Zannier \cite{BMZGeometric} on unlikely intersections in the algebraic
torus.

We also make use of  (very rudimentary)  tropical geometry. 
Say $K$ is the field of Puiseux series over $\IC$; it
is algebraically  closed and equipped with a surjective
valuation $\ordS : K \rightarrow\IQ \cup \{\infty\}$. 
Let $\cX$ be an irreducible subvariety defined over $K$ of the algebraic torus
$\IGm^n$.
The tropical variety $\trop{\cX}$ of $\cX$ is the  closure in
$\IR^n$ of 
\begin{equation*}
 \{(\ordS(x_1),\ldots,\ordS(x_n)) : (x_1,\ldots,x_n)\in \cX(K) \}. 
\end{equation*}

We need the following
two basic facts.

First, if $\cY\subset\cX$ is
an irreducible subvariety defined over $K$, then   $\trop{\cY}\subset
\trop{\cX}$; this follows directly from the definition.

Second, if $r=\dim \cX \ge 1$ and after permuting coordinates, the projection of 
 $\trop{\cX}$ to the first $r$ coordinates of $\IR^n$ contains
$\IQ^r$. 
We   prove this using basic algebraic geometry. 
After permuting coordinates the projection $\pi : \cX(K)\rightarrow
\IGm^r(K)$ onto the first $r$ coordinates contains a Zariski open and
dense subset of
$\IGm^r$. There exists a polynomial  $P\in
K[X_1^{\pm 1},\ldots,X_r^{\pm 1}]\ssm\{0\}$ 
such that any point of $\IGm^r(K)$ outside of the zero locus
 of $P$ lies in the  image of $\pi$. 
Let  $(a_1,\ldots,a_n)\in\IQ^r$ be arbitrary. For sufficiently general 
$(c_1,\ldots,c_r)\in\IGm^r(\IC)$ the polynomial $P$ does not vanish at
 $(c_1T^{a_1},\ldots,c_rT^{a_r}) \in \IGm^r(K)$.
This point has a pre-image under $\pi$  in $\cX(K)$
and the valuation of its first $r$ coordinates
are $a_1,\ldots,a_r$. This yields our claim.

Einsiedler, Kapranov and Lind's  \cite{EinKapLin}
 Theorem 2.2.5 on the structure of $\trop{\cX}$ 
can be used instead of this second property in the proof of 
Lemma \ref{lem:bmz} below.


We write $\langle \cdot,\cdot\rangle$ for the standard scalar product on $\IR^n$. 
If $\Omega$ is a subgroup of $\IZ^n$, then we set
\begin{equation*}
 \Omega^{\perp} = \{a\in \IZ^n : \langle a,\omega\rangle=0\text{ for
  all }\omega\in \Omega\}.
\end{equation*}

We write  $e_1,\ldots,e_n$ for  the standard basis elements of $\IR^n$
augmented by 
$e_0=0$.

\begin{lemma}
\label{lem:bmz}
Let $\Omega\subset\IZ^n$ be a subgroup of
rank $m\ge 2$
 for which the following property holds. 
If $(\alpha,\beta)\in\IZ^2\ssm\{0\}$ and if $i,j,k,l\in \{0,\ldots,n\}$
are pairwise distinct with $v = \alpha(e_i-e_j) + \beta ( e_k-e_l)$, then  $\Omega\not\subset (v\IZ)^{\perp}$.
Then there exist finitely many  subgroups
$\Omega_1,\ldots,\Omega_N\subset\Omega$ of rank at most $m-1$ such that if
$a\in\Omega\ssm \bigcup_{i=1}^N \Omega_i$ 
then $P_a$ does not vanish at any point of $S^1\ssm
\mu_\infty$.
\end{lemma}
\begin{proof}
We consider an irreducible component  $\mathcal{X}\subset\IGm^n$ of the
zero set of 
  \begin{equation}
\label{eq:defineX}
    1+X_1+\cdots + X_n \quad\text{and}\quad 1+X_1^{-1}+\cdots +
    X_n^{-1}. 
  \end{equation}
Then $\dim\mathcal{X}= n-2$
as these two polynomials are coprime in $\IC[X_1^{\pm
    1},\ldots,X_n^{\pm 1}]$.

Say $a=(a_1,\ldots,a_n)\in \Omega$ such that $P_a$ has a root $z\in
S^1\ssm\mu_\infty$. Then $1+z^{a_1}+\cdots + z^{a_n}=0$ by definition and after
applying complex conjugation we find 
\begin{equation*}
  1+z^{-a_1}+\cdots + z^{-a_n}=0.
\end{equation*}
So $z^a=(z^{a_1},\ldots,z^{a_n})\in {\mathcal X}(\IC)$ for one of the
irreducible components above.

But $z^a$ also lies in an algebraic subgroup of $\IGm^n$ of dimension
$1$. According to Bombieri, Masser, and Zannier's Theorem 1.7
\cite{BMZGeometric} there are two cases. Either $z^a$  lies in a finite
set that depends only on $\mathcal{X}$, and hence only on $n$, or
$z^a\in \cH(\IC)$ where $\cH\subset\IGm^n$  is an irreducible component of an
algebraic subgroup with  
\begin{equation*}
 \dim_{z^a} \mathcal{X}\cap \cH\ge \max\{1,\dim \mathcal{X} + \dim \cH - n +
 1\}
= \max\{1, \dim  \cH - 1\}. 
\end{equation*}
Moreover, in the second case $\cH$ comes from a finite set that
depends only on $n$, cf. the first paragraph of the proof of Theorem 1.7 \cite{BMZGeometric}
on page 26. 

In the first case we claim that $a$ must lie in one of finitely many
subgroups of $\Omega$ of rank $1 \le m-1$. Indeed, we may assume $a\not=0$. If $z'\in S^1\ssm\mu_\infty$ is a root
of some $P_{a'}$ with $a'\in \Omega$ and $z^{a}={z'}^{a'}$, then $a$
and $a'$ are linearly dependent as $z'$ is not a root of unity. 
Our claim follows as there are only finitely many possible $z^{a}$ in
this case. We add these rank $1$ subgroups to our collection of
$\Omega_i$.  

In the second case there is an irreducible component
$\cY\subset\mathcal{X}\cap \cH$
of positive dimension at least $\dim\cH -1$
that contains $z^a$.
In this case $n\ge 3$.
We recall that algebraic subgroups of $\IG_m^n$ are in bijection with
subgroups of $\IZ^n$, see Theorem 3.2.19 \cite{BG} for details. 
 Let $\Lambda\subset\IZ^n$ be a subgroup
from a finite set depending only on $n$ with rank $r=n - \dim \cH$ such that
$\cH$ is contained in the algebraic subgroup defined by all
\begin{equation}
\label{eq:defineH}
X_1^{b_1}\cdots X_n^{b_n}-1 \quad\text{where}\quad
(b_1,\ldots,b_n)\in \Lambda.
\end{equation}
Hence $z^{\langle a,b\rangle}=1$ for all $b\in\Lambda$.
As $z$ is not a root of unity we find 
$a\in \Lambda^{\perp}$. 

We would like to add $\Omega \cap \Lambda^{\perp}$ to our list of 
$\Omega_i$. However, we must 
ensure that its rank is at most $m-1$.
Once this is done, our proof is complete.

Suppose the rank does not drop, then 
$[\Omega : \Omega\cap \Lambda^{\perp}] \Omega \subset \Lambda^{\perp}$
and hence $\Omega \subset\Lambda^{\perp}$ because $\Lambda^{\perp}$ is primitive.
We now derive a contradiction from this situation by  analyzing
$\Lambda$.

Since all monomials 
  (\ref{eq:defineH}) vanish on $\cH$ we find that $\trop{\cH}$ lies in
$\Lambda^{\perp}\IR$,
  the vector subspace of $\IR^n$ generated by 
  $\Lambda^{\perp}$.
We  also  need to study $\trop{\cX}$. 
Say $(x_1,\ldots,x_n)\in \cX(\IC)$. Since it is a zero of the
 first polynomial in (\ref{eq:defineX}) and by
 the ultrametric
triangle inequality  the minimum among
$0, \ordS(x_1),\ldots,\ordS(x_n)$ is attained twice. 
The same argument applied to 
 the second polynomial in (\ref{eq:defineX}) shows that the
maximum is attained twice. 
Hence $\trop{\cX}$  is contained in the finite union
of the codimension $2$ vector subspaces of
$\IR^n$ defined by relations
\begin{equation*}
  \ordS(x_i) = \ordS(x_j), \quad \ordS(x_k) = \ordS(x_l) \quad \text{with} \quad \#\{i,j,k,l\} = 4
\end{equation*}
and 
\begin{equation*}
  \ordS(x_i) = \ordS(x_j), \quad \ordS(x_k) = 0 \quad \text{with} \quad \#\{i,j,k  \} = 3.
\end{equation*}

By the discussion before this lemma, 
$\trop{\cY}\subset \trop{\cX}\cap\trop{\cH} \subset \trop{\cX}\cap
\Lambda^{\perp}\IR$. Moreover, the projection of $\trop{\cY}$
 to some choice of  $\dim\cY$ distinct coordinates of $\IR^n$ contains $\IQ^{\dim \cY}$. 
So  $\Lambda^{\perp}\IR$ intersected with  one of the codimension
$2$ 
subspaces mentioned above must have dimension at least $\dim\cY \ge \dim\cH -1 = n-r-1$. 
Therefore, there exists $(\alpha,\beta)\in\IZ^2\ssm\{0\}$ and pairwise
distinct $i,j,k,l\in \{0,\ldots,n\}$ with 
$ v= \alpha( e_i - e_j)  + \beta (e_k- e_l) \in \Lambda$. 
Recall that $\Omega\subset\Lambda^{\perp}$. 
So $\Omega$ lies in the orthogonal complement of $v$, which contradicts the hypothesis of the lemma. 
\end{proof}

Suppose $n\ge 2$, let $\Omega\subset\IZ^n$ be a subgroup of rank at
least $2$ and let $a\in\IZ^n$. For a prime $p$  we define
\begin{equation}
\label{def:rhop}
  \rho_p(a;\Omega) = \inf \left\{|\omega| : \omega\in \Omega\ssm\{0\} 
\text{ and }\langle\omega,a\rangle \equiv 0 \imod p \right\}
\end{equation}
this is a well-defined real number  as the set is non-empty.

\begin{proposition}
\label{prop:norootsonunitcircle}
 Let $\Omega\subset\IZ^n$ be a subgroup of rank
 $m\ge 2$ that
 satisfies the following hypothesis. 
   If $(\alpha,\beta)\in\IZ^2\ssm\{0\}$ and if $i,j,k,l\in \{0,\ldots,n\}$
are pairwise distinct with $v = \alpha(e_i-e_j) + \beta ( e_k-e_l)$,
then  $\Omega\not\subset (v\IZ)^{\perp}$.
Then there
  exists a constant $c=c(\Omega)\ge 1$ with the following property.  Say
  $a\in\Omega$ and let $p$ be a prime with  $\rho_p(a;\Omega)\ge c$. 
Then there exist $t\in\IZ$ with $p\nmid t$ and
  $\omega\in\Omega$  such that  $a'=ta-p\omega\not=0$, 
  \begin{enumerate}
  \item [(i)] we have $|a'|\le cp^{1-1/m}$, and
  \item[(ii)]   the Laurent polynomial $P_{a'} \in \IZ[X^{\pm 1}]$ 
does not vanish at any point of $S^1\ssm\mu_\infty$.
  \end{enumerate}
\end{proposition}
\begin{proof}
The proposition follows from combining
 Lemma \ref{lem:bmz} with Dirichlet's Theorem from diophantine
 approximation.  
Indeed, we let $\Omega_1,\ldots,\Omega_N$
be the subgroups  of $\Omega$ from this lemma. If $N=0$ we set
$\Omega_1=\{0\}$. We will see how to choose $c$ below. 

  We fix a basis $(\omega_1,\ldots,\omega_m)$ of the abelian group
  $\Omega$. Then $a = \nu_1\omega_1+\cdots+\nu_m\omega_m$, where
$\nu_1,\ldots,\nu_m\in\IR$ are unique. 

If $p\ge 3$, then Dirichlet's Theorem, cf. Theorem 1B \cite{SchmidtLNM1467}, applied
to $\nu_1/p,\ldots,\nu_m/p,$ and $p-1>1$ yields
$t,\nu'_1,\ldots,\nu'_m\in\IZ$ such that $1\le t\le p-1$ and  $|t\nu_i/p -
\nu'_i|\le(p-1)^{-1/m}$. 
The  same conclusion holds for $p=2$.
 We set $\omega = \sum_{i=1}^m \nu'_i
\omega_i\in\Omega$.  Then 
$|a'/p| \le (|\omega_1|+\cdots+|\omega_m|)(p-1)^{-1/m} \le c p^{-1/m}$ 
where $a'=ta-p\omega$ for $c$ large enough.  This yields part (i). 

As each $\Omega_i$ has rank at most $m-1$ we have
 $\Omega \cap (\Omega_i)^{\perp}\not=0$ for all $i$. For each $i$
we fix a non-zero
$\omega^* \in \Omega\cap (\Omega_i)^{\perp}$ of minimal norm. 
If $\langle \omega^*,ta-p\omega\rangle = 0$, then 
$\langle \omega^*,a\rangle\equiv 0  \imod p$, since $p\nmid t$. So
$|\omega^*|\ge  \rho_p(a;\Omega)\ge c$ by hypothesis. We can avoid this outcome by fixing  $c$ large in terms of the
$\Omega\cap (\Omega_i)^{\perp}$. Thus $\langle \omega^*,a'\rangle
\not= 0$. This
implies $a' \not\in \Omega_i$ for all $1\le i\le N$ 
and in particular $a'\not=0$. Part (ii) follows from the conclusion
of Lemma \ref{lem:bmz}. 
\end{proof}

\section{Rational Points Close to a  Definable Set}
\label{sec:approxdefinable}

In this section we prove  Theorem \ref{thm:countrootsof1}.
To do this we temporarily adopt the language of o-minimal structures. Our main reference is van den Dries' book \cite{D:oMin} and his paper
with Miller \cite{DM:96}. We work exclusively with
the o-minimal structure $\IRan$ of restricted analytic functions. It contains the graph
of any
function  $[0,1]^n\rightarrow\IR$
that is the restriction of an analytic function  $\IR^n\rightarrow\IR$. 

The main technical tool in this section is a result of the author \cite{approxdefinable}
 which we cite in a special case below. 
We retain much of
the notation used in the said reference.
Roughly speaking, the result gives an upper bound for the number of rational
points of bounded height that are close to a subset of $\IR^n$ that is
definable in $\IRan$. 
Note that $\IRan$ is a polynomially bounded o-minimal
structure as required by this reference. 

For any subset $Z\subset\IR^n$   we write $\alg{Z}$ for the union of all
connected, semi-algebraic sets that are contained completely in
$Z$. For $\epsilon > 0$ we define
$\nbhd{Z,\epsilon}$ to be the set of $y\in\IR^n$ for which 
$|x-y|<\epsilon$ for some $x\in Z$. We recall that the height
$H(\cdot)$ was
defined in Section \ref{sec:notation}.

\begin{theorem}[Theorem 2 \cite{approxdefinable}]
\label{thm:approx2}
Let $Z\subset\IR^n$ be a closed set that is definable in $\IRan$ and let
$\epsilon>0$.
 There exist 
  $c=c(Z,\epsilon)\ge 1$ and
  $\theta=\theta(Z,\epsilon) \in (0,1]$ such that if $\lambda \ge\theta^{-1}$ then
\begin{equation*}
   \#\left\{q \in \IQ^n \ssm \nbhd{\alg{Z},T^{-\theta\lambda}} :  H(q)\le T
   \text{ and there is }
x\in Z\text{ with }|x-q|<T^{-\lambda} \right\}
\le cT^\epsilon
\end{equation*}
for all $T\ge 1$. 
\end{theorem}
\begin{proof}[Proof of Theorem \ref{thm:countrootsof1}]
For $n=1$ the theorem follows with $\lambda = 1$ and taking
 $c$ sufficiently large. Our proof is by induction on $n$ and we suppose
$n\ge 2$. We will choose $c\ge 1$ and $\lambda\ge 1$ in terms of $n$ and
$\epsilon$ during the argument. 

Let
\begin{equation*}
  Z = \left\{(x_1,\ldots,x_n) \in [0,1]^n : 1 +  e^{2\pi \sqrt{-1}
    x_1} + \cdots +  e^{2\pi \sqrt{-1} x_n} = 0 \right\}
\end{equation*}
which is compact and definable in $\IRan$.

Let us write 
$\zeta_j = e^{2\pi \sqrt{-1} q_j}$ for $j\in \{1,\ldots,n\}$ with $q_j \in
 \frac 1p\IZ\cap[0,1)$. 
Say $t$ is as in the set (\ref{eq:count}). For convenience, we
identify it with its  representative in
$\{0,\ldots,p-1\}$. Then 
\begin{equation}
\label{eq:smallsum2}
  \left| 1+  \zeta_1^{t} + \cdots + 
  \zeta_n^{t}\right|
<c^{-1}p^{-\lambda},
\end{equation}
 $(\zeta_1,\ldots,\zeta_n)$ has precise order $p$, and $t\not=0$. 
We claim that 
$\tilde{q}_t = (t q_1-\lfloor tq_1\rfloor,\ldots,t q_n-\lfloor tq_n\rfloor ) \in  \frac 1p \IZ^n \cap [0,1)^n$
  lies close to  $Z$. 
Indeed, a suitable version of \L ojasiewicz's Inequality, see 4.14.(2) \cite{DM:96},
implies
that the distance of this point to $Z$ is at most 
\begin{equation*}
 c_1 \left |1 + e^{2\pi \sqrt{-1} t q_1}+\cdots +
  e^{2\pi \sqrt{-1} t q_n}\right|^\delta
\le
c_1 c^{-\delta} p^{-\lambda\delta} \le c_1 c^{-\delta}. 
\end{equation*}
where $c_1>0$ and $\delta>0$ depend only on $Z$. 
We may assume $c_1 c^{-\delta}<1$,
 so
\begin{equation}
\label{eq:tildeqtx}
|\tilde{q}_t-x|<p^{-\lambda\delta}\quad
\text{for some $x\in Z$.}
\end{equation}

The vectors $\tilde{q}_0,\ldots,\tilde{q}_{p-1}$ are pairwise
distinct.
So it is enough to bound the number of $\tilde{q}_t$ with (\ref{eq:tildeqtx}). 

If $\lambda$ is sufficiently large, then $\lambda\delta \ge\theta^{-1}$ where
$\theta$ is provided by Theorem   \ref{thm:approx2} applied
to $Z$ and $\epsilon$. 
So there are at most $cp^\epsilon$ many 
$\tilde{q}_t$
 that are
not
in the $p^{-\theta\lambda\delta}$-tube around $\alg{Z}$.

To prove (\ref{eq:count}) we need only consider those
$t$ with 
 $| \tilde{q}_t-x'|<p^{-\theta\lambda\delta}$
for some $x'=(x'_1,\ldots,x'_n)\in
\alg{Z}$. 
The algebraic locus $\alg{Z}$ is well understood; see Ax's work
\cite{Ax:Schanuel} and how it is applied for example in 
the proof of Theorem 7 \cite{approxdefinable}. There exists
$\emptyset\not=J\subsetneq \{1,\ldots,n\}$ such that
$ 1 + \sum_{j\in J}  e^{2\pi \sqrt{-1} x'_j}=0$.
We subtract this from the partial sum over the coordinates of $t (q_1,\ldots,q_n)$ and get
\begin{equation*}
  1 + \sum_{j\in J}  \zeta_j^t = 
\sum_{j\in J} (e^{2\pi \sqrt{-1} t q_j}-e^{2\pi \sqrt{-1} x'_j}).
\end{equation*}
Thus
\begin{equation*}
  \left|1 + \sum_{j\in J}  \zeta_j^t\right|
\le\sum_{j\in J} \left|e^{2\pi \sqrt{-1} (t q_j-\lfloor tq_j \rfloor)}-e^{2\pi \sqrt{-1} x'_j}\right|
\le 2\pi(n-1)|\tilde{q}_t-x'|.
\end{equation*}
Hence there is a constant $c_2>0$ depending only on $n$ 
with
\begin{equation*}
  \left|1 + \sum_{j\in J}  \zeta_j^t \right| 
<  c_2 p^{-\theta\lambda\delta}. 
\end{equation*}

Recall $\# J\le n-1$. Now suppose $c'>0$ and $\lambda'  >0$ are the
constants from this theorem applied by induction to the terms in $J$.
We are free to assume that $\lambda$ satisfies
$2^{\theta\lambda\delta/2} \ge c_2c'$ and
$\theta\lambda\delta \ge 2 \lambda'$. Then 
$c_2p^{-\theta\lambda\delta} \le c_2 2^{-\theta\lambda\delta/2}
p^{-\theta\lambda\delta/2} \le {c'}^{-1}p^{-\lambda'}$
 as $p\ge 2$.
So
\begin{equation*}
  \left|1+ \sum_{j\in J}  \zeta_j^t 
  \right|<{c'}^{-1}p^{-\lambda'}. 
\end{equation*}

If  $1+\sum_{j\in J} \zeta_j\not=0$, then 
(\ref{eq:count}) follows from induction. 

On the other hand, if this sum vanishes, then its Galois conjugates
 $1+\sum_{j\in J}\zeta_j^t$ vanish for all
$t\in \IF_p^\times$.
By hypothesis, the normalized complementary sum
\begin{equation*}
1+  \sum_{j\in I\ssm\{j_0\}} \zeta_j\zeta_{j_0}^{-1}
\end{equation*}
is non-zero; here $I =  \{1,\ldots,n\}\ssm J$ and $j_0\in I$. We may apply
induction since $|I|-1 \le n-2$. Using (\ref{eq:smallsum2}) we find 
\begin{equation*}
\left|1 +  \sum_{j\in I \ssm \{j_0\}}
  (\zeta_j\zeta_{j_0}^{-1})^t\right|
=\left|1+ \zeta_1^t + \cdots + \zeta_n^t\right| < c^{-1} p^{-\lambda}. 
\end{equation*}
and there are at most $cp^\epsilon$ possibilities for $t$. 
\end{proof}

\section{Counting Small Sums of Roots of Unity}

Let $P_a$ be a lacunary Laurent polynomial as in (\ref{def:Pa}) with
$a\in\IZ^n$ where $n\ge 1$. 
The goal of this section is to 
bound the number of roots $\alpha$ of $P_a$ coming from $\scrB=\scrB(p,a,\lambda)$
defined in 
(\ref{def:scrB}). 
If $\alpha^p$ is close to $1$, then $\alpha$ is close to a root of
unity  $\zeta$ with $\zeta^p=1$. So $|P_a(\zeta)|$ will be
small. Thus $1+\zeta^{a_1}+\cdots + \zeta^{a_n}$ is small in modulus
where $a =(a_1,\ldots,a_n)$. 
This is where the counting result  proved in  Section \ref{sec:approxdefinable}
 comes into play. 

We make the first part of this approach precise in the next lemma. 

\begin{lemma}
\label{lem:zetaapproxalpha}
 Suppose $\alpha = re^{2 \pi \sqrt{-1}\vartheta}$ 
  with $r\in (0,1]$ and $\vartheta\in [0,1)$. 
If $p\ge 2$ is an integer with $|\alpha^p-1|\le 1/2$, then there
exists
$t\in\{0,\ldots,p\}$ such that 
\begin{equation}
\label{eq:zetaapproxalpha}
\left|\vartheta-\frac tp\right| \le \frac{1}{4\sqrt{2}} |\alpha^p-1|.
\end{equation}
If in addition $a\in\IZ^n$ and $P_a(\alpha)=0$, then 
\begin{equation}
\label{eq:zetaapproxalpha2}
\left|P_a\left(e^{2\pi \sqrt{-1}t/p}\right)\right|\le 5n |a|
|\alpha^p-1|. 
\end{equation}
\end{lemma}
\begin{proof}
Observe that $|z-1|\ge |z|^{1/2} |z/|z|-1|$ 
for all $z\in\IC\ssm\{0\}$, see Lemma 11.6.1 \cite{RahmanSchmeisser}. 
We  substitute
$z=\alpha^p$ to find
$|\alpha^p-1|\ge r^{p/2} |e^{2\pi \sqrt{-1}\vartheta p}-1|$. 
Now $1-r^p\le |r^p-1|\le|\alpha^p-1|$, so 
$r^p\ge 1-|\alpha^p-1|\ge 1/2$ by hypothesis. 
We find
\begin{equation}
\label{eq:alphapup}
  \left|e^{2\pi \sqrt{-1} \vartheta p}-1\right|\le \sqrt 2 |\alpha^p-1|. 
\end{equation}

Let $t$ be an integer with $|\vartheta p -t|\le 1/2$. Then $t\in
\{0,\ldots,p\}$ as $\vartheta\in [0,1)$. So
\begin{equation}
\label{eq:explb}
\left|e^{2\pi \sqrt{-1}\vartheta p}-1\right|=  
\left|e^{2\pi \sqrt{-1}(\vartheta p-t)}-1\right|\ge 4 |\vartheta p-t|
\end{equation}
by elementary geometry. 
Combining (\ref{eq:alphapup}) with (\ref{eq:explb}) and dividing by
$p\ge 2$
yields
 (\ref{eq:zetaapproxalpha}).

To prove the second claim 
we  set $\xi = e^{2\pi \sqrt{-1}t/p}$ and  estimate
\begin{alignat}1
\nonumber
  |\alpha-\xi| &=|re^{2\pi \sqrt{-1} \vartheta}-\xi|
\le |r-1|+|e^{2\pi \sqrt{-1}\vartheta}-\xi|
=|r-1|+|e^{2\pi \sqrt{-1}(\vartheta-t/p)}-1|
\\
\label{eq:distalphapthrootof1}
&\le |\alpha^p-1| + 2\pi |\vartheta-t/p| 
\le (1+{\pi}/{\sqrt{8}})|\alpha^p-1|;
\end{alignat}
where we used $|r-1|\le |r^p-1|\le |\alpha^p-1|$. 

We fix $e\in\IZ$ such that $X^e P_a(X)$ is a polynomial with non-zero
constant term. Then
\begin{equation*}
  |P_a(\xi)|  = |\xi^e P_a(\xi) - \alpha^e P_a(\alpha)|
  \le \sum_{k=0}^n |\xi^{a_k+e}-\alpha^{a_k+e}|
\end{equation*}
here $a = (a_1,\ldots,a_n)$ and $a_0=0$.
Some $a_k+e$ vanishes and we use $|\xi|=1$ and $|\alpha|\le 1$ to
find
\begin{equation*}
  |P_a(\xi)|   \le n \max_{0\le k\le n}\{ a_k + e \} |\xi-\alpha|
  \le 2n |a| |\xi-\alpha|.
\end{equation*}
We recall (\ref{eq:distalphapthrootof1}) to obtain
 (\ref{eq:zetaapproxalpha2}). 
\end{proof}

Next we show that many elements in $\scrB$ will lead to many different
roots of unity as given by the lemma above. 
The reason for this is that roots of lacunary
polynomials are nearly angularly equidistributed by a result
of Hayman, known already to
 Biernacki. 


\begin{lemma}
\label{lem:boundB}
Let $a\in\IZ^n$  and suppose $p\ge 2$ is an integer with
$|a|\le p$. If
$\lambda \ge 1$, then 
\begin{equation*}
\#\scrB\le 12n  \#\{ \zeta\in \mu_p:  |P_a(\zeta)|<5n|a|p^{-\lambda}
\}. 
\end{equation*}
\end{lemma}
\begin{proof}
We may assume $a\not=0$. 
Let us partition $\scrB=\scrB(p,a,\lambda)$ into 
$\scrB_{<1} = \{i\in \scrB : |\alpha_i|<1\}$ and 
$\scrB_{>1} = \{i\in \scrB : |\alpha_i|>1\}$.

We construct a map 
\begin{equation*}
\psi_{<1}:  \scrB_{<1} \rightarrow \{0,\ldots,p\}
\end{equation*}
in the following manner. 
If $i\in \scrB_{<1}$, then 
$\alpha_i = |\alpha_i|e^{2\pi \sqrt{-1}\vartheta}$
for $\vartheta\in [0,1)$, which depends on $i$, and we have
 $|\alpha_i^p-1|<
  p^{-\lambda}\le 1/2$. Lemma \ref{lem:zetaapproxalpha} yields $t\in\{0,\ldots,p\}$ with 
$|\vartheta-t/p|\le |\alpha_i^p-1|/\sqrt{32}<p^{-\lambda}/\sqrt{32}$. We set $\psi(i) = t$. 

We define
$\psi_{>1}:\scrB_{>1}\rightarrow\{0,\ldots,p\}$ in the same spirit. 
For if $i\in \scrB_{>1}$, then   $|\alpha_i|>1$ and there is 
$\vartheta\in [0,1)$ such that $\alpha_i=|\alpha_i|e^{-2\pi\sqrt{-1}\vartheta}$.  We apply the said lemma to
$\alpha_i^{-1}$
 and $P_{-a}$ to obtain $\psi(i) \in \{0,\ldots,p\}$ with $|\vartheta -
 \psi(i)/p|<p^{-\lambda}/\sqrt{32}$.

For $i\in \scrB_{<1}$ we get
\begin{equation*}
  \left|P_a(\xi^{\psi(i)})\right| < 5n|a|p^{-\lambda}
\end{equation*}
with $\xi = e^{2\pi \sqrt{-1}/p}$. If $i\in \scrB_{>1}$, the same
bound holds for
$|P_{-a}(\xi^{\psi(i)})|=|P_{a}(\xi^{-\psi(i)})|$. 

Now $\# \scrB\le 2 \#\scrB_{<1}$
or $\# \scrB\le 2 \#\scrB_{>1}$. We assume the former, the latter case
is dealt with similarly.

Let us fix $e\in\IZ$ such that $Q = X^e P(X)$ is a polynomial
with non-zero constant term. Then $\deg Q \le 2|a|$ and $Q$ has 
at most $n+1$ terms. 
Elements of $\scrB_{<1}$ that are in a fiber of $\psi$ come from roots of $Q$
that are in an open sector of the complex plane with angle
$4\pi p^{-\lambda}/\sqrt{32} = \pi p^{-\lambda}/\sqrt 2$.  
By Proposition 11.2.4 \cite{RahmanSchmeisser}, such an open sector contains at most
\begin{equation*}
 \frac{\deg(Q)}{2\sqrt 2 p^{\lambda}} + n+1
 \le \frac{2|a|}{2\sqrt{2} p^{\lambda}} + n+1
\le\frac{|a|}{\sqrt{2}p} + n+1 \le \frac{1}{\sqrt{2}}+n+1\le 3n
\end{equation*}
roots of $Q$, counting multiplicities; here we used
$\lambda\ge 1$ and $|a|\le p$.
Therefore,
\begin{equation*}
  \#\scrB_{<1} \le 3n\#\left\{ 0\le t\le p : |P_a(\xi^t)|< 5n|a| p^{-\lambda} \right \}   
\end{equation*}
We must compensate for the fact that we may be
 counting $1=\xi^0 = \xi^p$ twice, so
\begin{equation*}
  \#\scrB_{<1} \le 6n \#\left\{ 0\le t\le p-1 : |P_a(\xi^t)|< 5n|a| p^{-\lambda}  \right\}.
\end{equation*}
The lemma now follows from $\#\scrB\le 2\#\scrB_{<1}$. 
\end{proof}


\begin{proposition}
\label{prop:boundscrB}
For all $\epsilon > 0$  there exist constants $c=c(n,\epsilon)\ge 1$
and $\lambda=\lambda(n,\epsilon)\ge 1$ with the following property. 
Let $a\in\IZ^n$  and let $p$ be  a prime 
with
  $|a|\le p$ such that $P_a$ does not vanish at any point of $\mu_p$. Then 
  \begin{equation*}
    \#\scrB(p,a,\lambda)\le c p^{\epsilon}. 
  \end{equation*}
\end{proposition}
\begin{proof}
Say $c'$ and $\lambda'$ are from Theorem \ref{thm:countrootsof1}.
We apply Lemma \ref{lem:boundB} to a fixed  $\lambda\ge \lambda'+1 \ge
2$ that satisfies 
$2^{\lambda-\lambda'-1}\ge 5nc'$ and hence 
$p^{\lambda}\ge 5nc' p^{\lambda'+1}$.  

Say $\xi=e^{2\pi \sqrt{-1}/p}$  and $a=(a_1,\ldots,a_n)$. 
Now any $\zeta\in\mu_p$ with 
$|P_a(\zeta)|<5n|a|p^{-\lambda}$ equals $\xi^t$ for some $t\in\IF_p$
and
\begin{equation*}
  \left| 1+\xi^{a_1 t} + \cdots + \xi^{a_n t}\right| <
  5n|a|p^{-\lambda} \le {c'}^{-1} \frac{|a|}{p} p^{-\lambda'}
\le {c'}^{-1} p^{-\lambda'},
\end{equation*}
as $|a|\le p$. 
Recall that $1+\xi^{a_1}+ \cdots + \xi^{a_n}=P_a(\xi)\not=0$ for 
  $a$ as in the hypothesis. So by Theorem \ref{thm:countrootsof1} the
  number of possible $t$  is at most
$cp^\epsilon$ for $c$ sufficiently large in terms of $n$ and $\epsilon$.
\end{proof}

\section{Main Technical Result}

Suppose $n\ge 2$, let $\Omega\subset\IZ^n$ be a subgroup of rank at
least $2$, and let $a\in\IZ^n$. We set 
\begin{equation}
\label{def:rho}
  \rho(a;\Omega) = \inf \left\{|\omega| : \omega\in \Omega\ssm\{0\} 
\text{ and }\langle\omega,a\rangle = 0\right\}
\end{equation}
and recall (\ref{def:rhop}).

\begin{proposition}
\label{prop:main}
Let   $\Omega\subset\IZ^n$ be  a subgroup of rank $m\ge 2$
 that satisfies the hypothesis in Proposition
 \ref{prop:norootsonunitcircle}.
There exists a constant $c=C(\Omega)\ge 1$ with the following
property. 
Say $a\in \Omega$ and let $p$ be a prime with
$\rho_p(a;\Omega)\ge c$.
There exists $a'\in\Omega\ssm\{0\}$ with 
$a'=at-p\omega$, where $t\in\IZ$ is coprime to $p$ and $\omega\in\Omega$,
such that $|a'|\le c p^{1-1/m},\Delta_p(a')=\Delta_p(a)\not=0$, 
\begin{equation}
\label{eq:mainpropconclusion}
\rho(a';\Omega)\ge \rho_p(a;\Omega), \quad\text{and}\quad
 \frac{1}{p-1} \log\Delta_p(a') = \ma{P_{a'}}+
 O\left({p^{-\frac{1}{2m}}}\right);
\end{equation} 
here the constant implicit in $O(\cdot)$ depends only on $\Omega$.
\end{proposition}
\begin{proof}
Let $\omega_1\in \Omega\ssm\{0\}$, then $\langle p\omega_1,  a\rangle
\equiv 0 \imod p$, hence
$\rho_p(a;\Omega)\le p |\omega_1|$. 
By increasing $c$ we may assume that $p$ is larger than a prescribed
constant. 
 Say $c_1\ge 1$ is the constant from
   Proposition \ref{prop:norootsonunitcircle}; we may suppose $c\ge
   \max\{3,c_1\}$. 
There
 is $t\in\IZ$ not
  divisible by $p$ and $\omega\in \Omega$ such that $a'=ta-p\omega$
  satisfies $0<|a'|\le c_1 p^{1-1/m}$
 and the only roots of $P_{a'}$ on the unit circle are roots of
unity. We may assume $|a'|<(p-1)/2$ by increasing $c$.

Suppose $\langle\omega_0,a'\rangle=0$ with
$\omega_0\in\Omega\ssm\{0\}$. Then
 $\langle\omega_0,a\rangle\equiv 0 \imod p$ as $t$ and $p$ are
coprime. So 
$|\omega_0|\ge \rho_p(a';\Omega)$ by hypothesis. This implies 
the inequality in (\ref{eq:mainpropconclusion}).  

Let $e\in\IZ$ such that $X^eP_{a'}$ is a polynomial with non-zero
constant part, then $\deg(X^e P_{a'})\le 2|a'|<p-1$. Since $X^eP_{a'}$
has integral coefficients we find $P_{a'}(\zeta)\not=0$ if $\zeta$ has
order $p$. Clearly, we also have $P_{a'}(1)=n+1\not=0$. 
 So $P_{a'}$ does not vanish at any point of 
 $\mu_p$.

We apply Proposition \ref{prop:boundscrB} to $\epsilon =1/(2m), a'$, and $p$ to conclude 
$\#\scrB(p,a',\lambda)\le c_2p^{1/(2m)}$ for constants $c_2,\lambda\ge 1$ that
depend only on $m$ and $n$.

We use this bound in the estimate from Proposition
\ref{prop:firstbound} applied to $a'$, to get
$\Delta_p(a')\not=0$ and
\begin{equation*}
  \left|\frac{1}{p-1} \log\Delta_p(a') - \ma{P_{a'}} \right| 
\le c_3 \frac{|a'|\log p}{p} \left(\lambda + c_2p^{\frac{1}{2m}}\right)
\le c_4 p^{-\frac{1}{2m}}
\end{equation*}
where $c_3$ and $c_4$ depend only on $m,n,$ and $\epsilon$.

Now recall that $a' = ta-p\omega$, so 
$\Delta_p(a')=\Delta_p(ta) = \Delta_p(a)\not=0$, by Lemma \ref{lem:deltasymmetry}, parts
(i) and (ii), respectively.
This completes the proof. 
\end{proof}

Let $A\in \mat{mn}{\IZ}$ be a matrix with entries $a_{ij}$. We define
\begin{equation}
\label{def:PA}
  P_A = 1+\sum_{j=1}^n X_1^{a_{1j}} X_2^{a_{2j}}\cdots X_m^{a_{mj}} 
\in\IZ[X_1^{\pm 1},\ldots,X_m^{\pm 1}].
\end{equation}
If we consider $a\in\IZ^n$ as a $1\times n$ matrix and identify $X_1$
with $X$, then the definitions (\ref{def:Pa}) and (\ref{def:PA})
coincide. 

Now suppose that $\Omega\subset\IZ^n$ is a subgroup of rank $m\ge 1$
and assume that  the rows of $A$ are a basis of $\Omega$. Then the
value $\ma{P_A}$ is independent of the choice of basis. Indeed, if the
rows of $B\in\mat{mn}{\IZ}$ constitute another basis of $\Omega$, then 
$A=UB$ with $U\in \gl{m}{\IZ}$. The  Mahler
measure  is known to be invariant under a change of coordinates by
$U$, i.e. $\ma{P_A} = \ma{P_{UB}}=\ma{P_B}$, cf. Corollary 8 in Chapter
3.4 \cite{Schinzel}. 
So $\ma{P_A}$ depends only on $\Omega$ and we write
\begin{equation}
  \label{def:omega}
 \ma{\Omega} = \ma{P_A}
\end{equation}
for any $A$ as before. 

\begin{theorem}
\label{thm:main}
Let   $\Omega\subset\IZ^n$ be a subgroup of rank $m\ge 2$
 that satisfies the hypothesis in Proposition
 \ref{prop:norootsonunitcircle}
and say $\epsilon > 0$. 
Suppose $a\in \Omega$ and $p$ is a prime such that
 $\rho_p(a;\Omega)$ is sufficiently large in terms of $\Omega$. Then
 $\Delta_{p}(a)\not=0$ and
\begin{equation}
  \label{eq:thmmainlimit}
  \frac{1}{p-1} \log \Delta_{p}(a) = 
\ma{\Omega} + O\left(\rho_{p}(a;\Omega)^{-\frac{1}{4n}+\epsilon}\right)
\end{equation}
as $\rho_{p}(a;\Omega)\rightarrow\infty$
where the implicit constant depends only on $\Omega$ and $\epsilon$. 
\end{theorem}
\begin{proof}
If $\rho_p(a;\Omega)$ is  sufficiently large, then by
 Proposition \ref{prop:main} we have $\Delta_p(a)\not=0$ and 
 $a'\in \Omega$ with the stated properties. 

Let us fix a basis $\omega_1,\ldots,\omega_m$ of $\Omega$ and write
 $A\in\mat{mn}{\IZ}$ for the matrix with   rows $\omega_1,\ldots,\omega_m$.
  We fix a tuple of independent elements 
  $(\omega^*_1,\ldots,\omega^*_m)$ in $\Omega$ and an integer $d\ge 1$ with
$\langle \omega^*_j,\omega_l\rangle = 0$ if $j\not=l$ and 
$\langle \omega^*_j,\omega_j\rangle =d$ for all $1\le j,l\le m$. 

It remains to check that  $\ma{P_{a'}}$ converges to
$\ma{P_A} = \ma{\Omega}$. 
Observe
\begin{equation*}
  P_{a'} = P_A(X^{\nu_1},X^{\nu_2},\ldots,X^{\nu_m})
\end{equation*}
where $\nu=(\nu_1,\ldots,\nu_m)\in\IZ^m$ is  determined by
$a' = \nu_1 \omega_1 + \cdots + \nu_m\omega_m$. 

Our quantitative version of Lawton's Theorem,  Theorem
\ref{thm:lawtonquant} in the self-contained appendix, 
relies on
$\rho(\nu;\IZ^m)$.
Say $(\lambda_1,\ldots,\lambda_m)\in\IZ^m\ssm\{0\}$
has norm $\rho(\nu;\IZ^n)$
with $\sum_{j=1}^m \lambda_j\nu_j=0$.
Then $\langle \sum_{j=1}^m\lambda_j\omega^{*}_j,a'\rangle = 0$ and
hence
\begin{equation*}
\rho_p(a;\Omega)\le\rho(a';\Omega)\le \left|\sum_{j=1}^m
\lambda_j\omega^{*}_j\right|\le \rho(\nu;\IZ^m) (|\omega^{*}_1|+\cdots+|\omega^{*}_m|)
\end{equation*}
 where we used (\ref{eq:mainpropconclusion}).
The theorem follows as $P_A$ has at most $n+1$  and at
least $m+1\ge 2$ non-zero terms. 
\end{proof}

\section{Applications to Gaussian Periods}

Using our method we  prove the following theorem from which
we will deduce two applications. 
Suppose $n\ge 2$. Recall that $\ma{\Omega}$ was defined in  (\ref{def:omega}). 

A place of a number field is an absolute value on the said number
field whose restriction to $\IQ$ coincides with  the standard complex absolute
value or the $p$-adic absolute value for a prime $p$ taking the value
$p^{-1}$ at $p$. 

\begin{theorem}
\label{thm:mainapp}
Let $K$ be a number field, let $\alpha_1,\ldots,\alpha_n\in K$, and
define the subgroup
  \begin{equation*}
    \Omega = \left\{ (b_1,\ldots,b_n) \in\IZ^n :  b_1\alpha_1+\cdots+b_n\alpha_n =
    0 \right\}^{\perp}
  \end{equation*}
of rank $m$. Let $\epsilon > 0$.
We  suppose $m\ge 2$ and 
that the following hypothesis holds.
The $0=\alpha_0,\alpha_1,\ldots,\alpha_n$ are pairwise
    distinct
and $(\alpha_i-\alpha_j)/(\alpha_k-\alpha_l)\not\in\IQ$
for all pairwise distinct $i,j,k,l\in \{0,\ldots,n\}$. 
Let $p$ be a prime, $v_0$  a place of $K$ 
that extends the $p$-adic absolute value,
and $e(v_0)$ the ramification index of $K/\IQ$ at $v_0$. 
Suppose
$(a_1,\ldots,a_n) \in\IZ^n$ with $|a_i-\alpha_i|_{v_0}<1$ for all
$i\in \{1,\ldots,n\}$, then
\begin{equation}
\label{eq:mainapp}
  \frac{1}{p-1} \sum_{t=1}^{p-1} \log \left|1 + \zeta^{t a_1} + \cdots +
  \zeta^{t a_n}\right| = \ma{\Omega} + 
O\left(p^{-\frac{1}{4n[K:\IQ]e(v_0)}+\epsilon}\right) \quad\text{where}\quad
\zeta = e^{2\pi\sqrt{-1}/p}
\end{equation}
as $p\rightarrow \infty$ and the implicit constant depends only on $\alpha_1,\ldots,\alpha_n,$
and $\epsilon$; in particular, the  logarithm
is well-defined for all large $p$. 
\end{theorem}
\begin{proof}
Observe that $\Lambda = \{(b_1,\ldots,b_n) \in\IZ^n :
b_1\alpha_1+\cdots + b_n\alpha_n=0\}$ is a primitive subgroup of
$\IZ^n$ of rank $r=n-m$.
If $r\ge 1$ and if $M \in \mat{rn}{\IZ}$ is a matrix whose rows are a basis of
$\Lambda$, then
\begin{equation*}
  0 \rightarrow p\Omega \xrightarrow{\text{inclusion}} p\IZ^n
  \xrightarrow{\text{multiplication by $M$}}  p\IZ^r \rightarrow 0
\end{equation*}
is a short exact sequence. By  the Snake Lemma we
find that the image of $\Omega$ in $\IF_p^n$ equals the kernel of
multiplication by $M$ taken as an endomorphism $\IF_p^n\rightarrow
\IF_p^r$.

Say $a=(a_1,\ldots,a_n)\in\IZ^n$ is as in the hypothesis. Then the left-hand side of (\ref{eq:mainapp}) equals
$\frac{1}{p-1}\log \Delta_p(a)$.
This mean is invariant under translating $a$ by a vector in $p\IZ^n$, cf.
 Lemma \ref{lem:deltasymmetry}(i). 
For all
$(b_1,\ldots,b_n)\in \Lambda$ we find
\begin{equation*}
|  a_1 b_1+\cdots + a_nb_n|_{v_0} = 
|  (a_1-\alpha_1) b_1+\cdots + (a_n-\alpha_n) b_n|_{v_0} 
\le \max_{1\le i\le n} |a_i-\alpha_i|_{v_0} <1. 
\end{equation*}
By the previous paragraph we may assume, after adding an element of
$p\IZ^n$ to $a$, i.e. without loss of generality, 
that $a\in\Omega$.   The current theorem  will follow from
 Theorem \ref{thm:main}.

As $\Lambda^{\perp\perp}=\Lambda$, we find that the hypothesis on
$\Omega$ implies the hypothesis on $\Omega$ 
 in Proposition \ref{prop:norootsonunitcircle}.

To estimate $\rho_p(a;\Omega)$  say
$\omega=(\omega_1,\ldots,\omega_n)\in\Omega\ssm\{0\}$ with
$\langle\omega,a\rangle\equiv 0 \imod p$ and $|\omega|
=\rho_p(a;\Omega)$.  
We define $\Delta = \omega_1 \alpha_1+\cdots + \omega_n \alpha_n
\in K$. Observe that $\Delta \not=0$ since
$\Lambda\cap\Lambda^{\perp}=\{0\}$. 
We estimate
\begin{equation*}
  |\Delta|_{v_0}  = |\omega_1 ( \alpha_1-a_1)+\cdots
+\omega_n(\alpha_n - a_n) + \omega_1 a_1 + \cdots + \omega_n a_n|_{v_0}
\le p^{-1/e(v_0)}.
\end{equation*}

If $v$ is an archimedean place of $K$, then the triangle inequality
yields $|\Delta|_v \le n |\omega|
\max\{|\alpha_1|_v,\ldots,|\alpha_n|_v\}$.
For any non-archimedean place $v$ of $K$ we get a stronger bound due
to the ultrametric triangle inequality, i.e.
$|\Delta|_v \le 
\max\{|\alpha_1|_v,\ldots,|\alpha_n|_v\}$.

We  take the product of $|\Delta|_v\not=0$ over all places with the appropriate
multiplicities and use the local bounds above in combination with the
product formula  to obtain
\begin{equation*}
  1\le \left(n|\omega|e^{h(\alpha_1)+\cdots+h(\alpha_n)}\right)^{[K:\IQ]} p^{-1/e(v_0)}.
\end{equation*}
So (\ref{eq:mainapp}) follows from (\ref{eq:thmmainlimit})
as $|\omega| = \rho_p(a;\Omega)$.
\end{proof}

 The following corollary generalizes Theorem \ref{thm:sprime} to
subgroups of  $\IF_p^\times$ of odd order. 

\begin{corollary}
\label{cor:myerson}
  Let $f\ge 3$ be an odd integer and $\epsilon > 0$. Say $\xi$ is a root of unity of order $f$ and define
  \begin{equation*}
    \Omega = \left\{ (b_1,\ldots,b_{f-1})\in\IZ^{f-1} : b_1 (\xi-1) + b_2
    (\xi^2-1) + \cdots + b_{f-1}(\xi^{f-1}-1)=0\right\}^{\perp}.
  \end{equation*}
For any prime $p$ 
  with $p\equiv 1 \imod {f}$ let  $G_p\subset\IF_p^\times$ be the subgroup
  of order $f$. 
Then
  \begin{equation*}
    \frac{1}{p-1} \sum_{t=1}^{p-1} \log\left| \sum_{g\in G_p} \zeta^{tg}
    \right| = \ma{\Omega} + O\left(p^{-\frac{1}{4(f-1)\varphi(f)}+\epsilon}\right)
\quad\text{where}\quad \zeta= e^{2\pi\sqrt{-1}/p}
  \end{equation*} 
as $p\rightarrow\infty$ where the implicit constant depends only on
$f$ and $\epsilon$; in particular, the logarithm is well-defined
for all large $p$. 
\end{corollary}

We will prove this corollary further down.
Here we treat
  the hypothesis on $\Omega$  in Theorem \ref{thm:mainapp}
for roots of unity. 

\begin{lemma}
\label{lem:rootofunity}
  Let $\zeta_1,\zeta_2,\zeta_3,\zeta_4$ be pairwise distinct roots of unity.
If $(\zeta_1-\zeta_2)/(\zeta_3-\zeta_4)\in\IQ$, then 
there exist distinct $i,j\in \{1,2,3,4\}$ with
$\zeta_i=-\zeta_j$.
\end{lemma}
\begin{proof}
  Set $\eta = \zeta_3\zeta_4^{-1},\xi =\zeta_1\zeta_2^{-1},$ and
  $\zeta = \zeta_2\zeta_4^{-1}$ and define
  \begin{equation*}
    x = \frac{\zeta_1-\zeta_2}{\zeta_3-\zeta_4} = 
\zeta \frac{\xi-1}{\eta-1}.
  \end{equation*}
We assume $x\in \IQ$ and, after possibly swapping $\zeta_1$ and
$\zeta_2$, also $x>0$. 

 Note that the number field 
 $K=\IQ(\zeta_1,\zeta_2,\zeta_3,\zeta_4)$
 has only complex embeddings 
 as it contains at least  4 roots of unity.
 So 
 the $K/\IQ$-norm $\no{\cdot}{K/\IQ}$ is never negative. 
 We have
 \begin{equation}
\label{eq:normx}
  x^{[K:\IQ]} = \no{x}{K/\IQ} =
  \frac{\no{\xi-1}{K/\IQ}}{\no{\eta-1}{K/\IQ}}.
\end{equation}

We now divide into four cases, depending on whether the orders of
$\eta$ and $\xi$  are prime powers or not. 

First suppose that neither $\eta$ nor $\xi$ has order a prime power.
Then $\eta-1$ and $\xi-1$ are units and hence $x=1$ by
(\ref{eq:normx}).
We obtain the vanishing sum of roots of unity
\begin{equation}
\label{eq:sumrootof1}
\eta   - \zeta\xi + \zeta -1 =0.
\end{equation}

If a non-trivial subsum vanishes,  then
\begin{equation*}
\eta -\zeta\xi = \zeta-1=0\quad\text{or}\quad
\eta +\zeta = -\zeta\xi-1 = 0 \quad\text{or}\quad 
  \eta -1 =-\zeta\xi +\zeta = 0.
\end{equation*}
The first and third cases are impossible as $\zeta\not=1$ and $\eta\not=1$.
Thus $\zeta\xi = -1$ and this implies $\zeta_1=-\zeta_4$, as desired.

If no non-trivial subsum vanishes, then 
 Mann's Theorem 1 \cite{Mann}  implies 
$\eta^6 = \xi^6 =
\zeta^6=1$. 
In the current case,   $\eta$ and $\xi$ must have precise order $6$. 
If $\eta=\xi$, then (\ref{eq:sumrootof1}) implies $\eta=1$ or $\zeta=1$ which
contracts the hypothesis. 
Hence $\eta\not=\xi$ and thus 
$\eta\xi=1$.
When combined with (\ref{eq:sumrootof1}) we have
\begin{equation}
\label{eq:zeta1zeta4minus1}
   \zeta\xi =  \frac{\eta-1}{\xi-1} \xi
= \frac{\eta\xi-\xi}{\xi-1} = \frac{1-\xi}{\xi-1}=-1,
\end{equation}
and again $\zeta_1=-\zeta_4$. 

Now suppose that $\xi$ has order $p^e$ with $p$  a prime and
$e\ge 1$ and that the order of $\eta$ is not a prime power. 
Then $\eta-1$ remains a unit but now
 $\no{\xi-1}{K/\IQ} = p^{[K:\IQ(\xi)]}$, so $x^{\varphi(p^e)} = p$. As $x$ is rational, we must have $\varphi(p^e)=1$, so
$p^e=2$ 
and thus $\zeta_1\zeta_2^{-1}=\xi=-1$, which completes this case.

Similarly, if $\eta$ has prime power order and $\xi$ does not, then 
we apply the argument from the last paragraph to $x^{-1} = \zeta^{-1}
(\eta-1)/(\xi-1)$ and conclude 
$\zeta_3\zeta_4^{-1}=\eta =-1$, as desired.

Finally, suppose $\eta$ has order $q^{e'}$ and $\xi$ has order $p^{e}$,
here $p$ and $q$ are primes and $e,e'\ge 1$. Now (\ref{eq:normx})
implies
  $x=p^{1/\varphi(p^{e})}q^{-1/\varphi(q^{e'})} \in\IQ$. 
If $p\not=q$, then $p^e=q^{e'}=2$ and so $\eta=\xi=-1$, as desired.
So say $p=q$, hence $x=p^{1/\varphi(p^{e})-1/\varphi(p^{e'})}\in\IQ$ and it
follows that
\begin{equation*}
\frac{1}{(p-1)p^{e-1}} - \frac{1}{(p-1)p^{e'-1}} \in\IZ
\end{equation*}
and this entails $e=e'$. We find $x=1$ and are thus back in the
situation of the first case except that $\eta$ and $\xi$ now have
prime power order.
We proceed similarly. 
If a non-trivial subsum in (\ref{eq:sumrootof1}) vanishes, then
again $\zeta_1=-\zeta_4$. Otherwise Mann's Theorem yields
$\eta^6=\xi^6=1$. This time $\eta$ and $\xi$ have equal order 
which is $2$ or $3$. If the common order is $2$ then we are
done. Else wise it is $3$ and again we must have $\eta\not=\xi$ which
again implies
$\eta\xi=1$. We conclude
 $\zeta_1=-\zeta_4$ as in (\ref{eq:zeta1zeta4minus1}).
\end{proof}

\begin{proof}[Proof of Corollary \ref{cor:myerson}]
  Let $\xi$ be a root of unity of order $f$. We set 
  \begin{equation*}
    \alpha_1 = \xi-1, \quad \alpha_2 = \xi^2-1,\quad\ldots\quad \alpha_{f-1}=\xi^{f-1}-1.
  \end{equation*}
Our aim is to apply Theorem \ref{thm:mainapp} to $K =\IQ(\xi)$ and
$n=f-1\ge 2$. Say $\Omega$ is as in the said theorem. 

Observe that
\begin{equation*}
  \alpha_1\IZ + \cdots + \alpha_n \IZ = 
(\xi-1) (\IZ + (\xi+1)\IZ + \cdots + (\xi^{f-2}+\cdots + \xi +1)\IZ)
 = (\xi-1)\IZ[\xi]
\end{equation*}
as $\xi^{f-1} = -(\xi^{f-2}+\cdots +\xi+1)$. This group has rank
$m=\varphi(f)$
and therefore $\Omega$ also has rank $m\ge 2$.


We shall show that the hypothesis on $\Omega$ in Theorem \ref{thm:mainapp} is satisfied. 
Indeed, if $i,j,k,l\in \{0,\ldots,n\}$ are pairwise distinct, then 
\begin{equation*}
  \frac{\xi^i - \xi^j}{\xi^k-\xi^l}\in\IQ
\end{equation*}
implies that $\xi$ has even order by Lemma \ref{lem:rootofunity}. This
contradicts our hypothesis on $f$.

Finally, observe that $p\equiv 1 \imod{f}$ means that $p$ splits
completely in $K$. For such $p$ there is $(a_1,\ldots,a_n)$ as
above (\ref{eq:mainapp}) where $v_0$ is any place of $K$ extending the
$p$-adic absolute value. Here $e(v_0)=1$. 
Now $1,1+a_1,\ldots,1+a_n$ is a complete set of representatives
of the subgroup of $\IF_p^\times$ of order  $f=n+1$. 
The corollary follows as
\begin{equation*}
\left| 1+\zeta^{ta_1}+\cdots +
  \zeta^{ta_n}\right| = 
 \left| \zeta^t+\zeta^{t(1+a_1)}+\cdots +
  \zeta^{t(1+a_n)}\right|
\end{equation*}
for all $1\le t\le p-1$. 
\end{proof}

\begin{proof}[Proof of Theorem \ref{thm:sprime}]
We set $\Omega$ as in the proof of Corollary \ref{cor:myerson}. Its
rank is $\varphi(f)=f-1$ as $f$ is a
prime by hypothesis. 
Since $\Omega$ is a primitive subgroup of $\IZ^{f-1}$ we conclude
$\Omega=\IZ^{f-1}$. In the definition (\ref{def:omega}) of $\ma{\Omega}$ we may take $A$
to equal the $(f-1)\times(f-1)$ unit matrix. The resulting logarithmic
Mahler
measure is that of $1+X_1+\cdots+X_{f-1}$. 
\end{proof}

\begin{proof}[Proof of Corollary \ref{cor:finitelymanyunits}]
Observe that $\sum_{g\in G}e^{2\pi \sqrt{-1} g/p}$ is an algebraic integer. 
  In view of Theorem \ref{thm:sprime} it suffices to 
verify $\ma{1+X_1+\cdots + X_{f-1}}\not=0$. 
The higher dimensional version of
 Kronecker's Theorem 
classifies integral polynomials whose 
 logarithmic Mahler measure vanishes, see  work of
Boyd \cite{Boyd:Kronecker}, Lawton \cite{Lawton77}, and
 Smyth \cite{Smyth:Kronecker}. As our
$1+X_1+\cdots +X_{f-1}$ is irreducible and $f\ge 3$ we easily deduce
from this classification
 that its logarithmic Mahler measure is non-zero.
\end{proof}

\appendix
\section*{Appendix On Lawton's Theorem}

In this appendix we provide a rate of convergence for the following
theorem of  Lawton.
The arguments here do not rely on  the rest of the paper. 
Say $n\ge 1$ is an integer. Recall that the definition of $\rho$ is given in  (\ref{def:rho}).  

\begin{theorem}[Lawton, Theorem 2  \cite{Lawton}]
Suppose $P\in\IC[X_1^{\pm 1},\ldots,X_n^{\pm 1}]\ssm\{0\}$.
If $a\in\IZ^n$, then
$\ma{P(X^{a_1},\ldots,X^{a_n})} = \ma{P}+
o(1)$ as $\rho(a;\IZ^n)\rightarrow\infty$. 
\end{theorem}

We closely following Lawton's approach but keep track of
estimates to obtain the following refinement. 

 \begin{theorem}
\label{thm:lawtonquant}
Suppose $P\in\IC[X_1^{\pm 1},\ldots,X_n^{\pm 1}]\ssm\{0\}$ has $k\ge
   2$ non-zero terms and let $\epsilon >0$.  
For $a=(a_1,\ldots,a_n)\in\IZ^n$ we have
\begin{equation*}
\ma{P(X^{a_1},\ldots,X^{a_n})} = \ma{P}+
O\left(\rho(a;\IZ^n)^{-\frac{1}{4(k-1)}+\epsilon}\right)
\end{equation*}
as $\rho(a;\IZ^n)\rightarrow\infty$ where the implied constant depends
on $n,P,$ and $\epsilon$. 
 \end{theorem}

An important tool is an estimate on the measure of the subset of the
unit circle, where a polynomial takes small values. 
We use $\vol{\cdot}$ to denote the Lebesgue  measure on $\IR^n$. 
For $P\in \IC[X_1^{\pm 1},\ldots,X_n^{\pm 1}]$ and $\sv > 0$ we
define
\begin{equation*}
  S(P,\sv)  =
\{ x \in [0,1)^n : |P(e(x))|< \sv \}
\end{equation*}
where  $e(x)=(e^{2\pi \sqrt{-1}x_1},\ldots,e^{2\pi \sqrt{-1}x_n})$ 
for $x=(x_1,\ldots,x_n)\in \IR^n$. 

\begin{lemma}[Lawton, Theorem 1 \cite{Lawton}]
\label{lem:lawtonthm1}
For $k\ge 1$ there exists a constant $C_k>0$ with the following
property. 
If $P\in \IC[X]$ is a monic polynomial with $k$ non-zero
  terms and $\sv >0$, then
   $ \vol{ S(P,\sv) }
\le C_k \sv^{1/\max\{1,k-1\}}$
\end{lemma}
The reference covers the important case $k\ge 2$.
If $k=1$, then $|P(e(x))|=1$ for all
$x\in\IR^n$ and the claim is clear with $C_1 = 1$.

\begin{lemma}
\label{lem:volSPep}
Assume $P\in \IC[X_1^{\pm 1},\ldots,X_n^{\pm 1}]
\ssm\{0\}$ has $k\ge 1$ non-zero terms. 
\begin{enumerate}
\item [(i)]
If for all $\sv  \in (0,1]$ then    $\vol{S(P,\sv)}
=O\left( \sv^{1/(2\max\{1,k-1\})}\right)$
where the constant in $O(\cdot)$ depends only on $P$. 
\item[(ii)] If $q>0$ then
$\int_{[0,1)^n}\bigl|\log |P(e(x))|\bigr|^q dx$ is well-defined and finite. 
\end{enumerate}
\end{lemma}
\begin{proof}
  Our proof of (i) is by induction on $n$. We may assume that $P$ is a
  polynomial.  

If $n=1$, then the lemma follows from Lawton's Theorem after
normalizing $P$. 

Say $n\ge 2$.  There is nothing to prove if $k=1$, so suppose $k\ge 2$.
We may also assume that $P$ is a polynomial, hence $P = P_0 X_n^{d} + \cdots + P_d$
where $P_0,\ldots,P_d \in \IC[X_1,\ldots,X_{n-1}]$ and $P_0\not=0$. 
We abbreviate 
$\Sigma  = S(P_0,\sv^{1/2})\subset [0,1)^{n-1}$. 
If  $x'\in\IR^{n-1}$ and $P_0(e(x'))\not=0$, then 
$P(e(x'),X)/P_0(e(x'))$ is a monic polynomial in $X$. 
 Fubini's Theorem implies
\begin{alignat*}1
  \vol{S(P,\sv)} &= 
\int_{\Sigma}  \vol{S(P(e(x'),X),\sv)}
    dx'+
\int_{[0,1)^{n-1}\ssm \Sigma}  \vol{S(P(e(x'),X),\sv)}
    dx' \\
&\le 
\vol{\Sigma}+
\int_{[0,1)^{n-1}\ssm \Sigma}  \vol{S(P(e(x'),X)/P_0(e(x'))),\sv/|P_0(e(x'))|}
    dx' \\
&\le 
\vol{\Sigma}+
\int_{[0,1)^{n-1}\ssm \Sigma}  \vol{S(P(e(x'),X)/P_0(e(x'))),\sv^{1/2}}
    dx' \\
&\le 
\vol{\Sigma}+
C_k \sv^{1/(2k-2)}
\end{alignat*}
since $\vol{S(P(e(x'),X),\sv)}\le 1$ and 
by Lemma \ref{lem:lawtonthm1}.

By this lemma applied by induction to $P_0\in \IC[X_1,\ldots,X_{n-1}]\ssm\{0\}$ we conclude
$\vol{\Sigma} \le c(P_0) \sv^{1/(2k-2)}$. 
This yields part (i). 

The statement in (ii) is possible known, we give a proof based on (i). 
For an integer $m\ge 0$ and $x\in\IR^n$ we define
$p_m(x) = \min \{m, |\log |P(e(x))||^q\}\ge 0$ which is interpreted as $m$
if $P(e(x))=0$. Then $p_m$ is a non-decreasing sequence of continuous
functions on $[0,1)^n$.
We set $I_m = \int_{[0,1)^n} p_m(x)dx$. By the Monotone Convergence
  Theorem it suffices to prove that the non-decreasing sequence $(I_m)_{m\ge 1}$
  converges.  
 For all  sufficiently large $m$ we
have $|P(e(x))|\le e^{m^{1/q}}$ if $x\in [0,1)^n$. 
Observe that $p_m$ equals $m$ on $S(P,e^{-m^{1/q}})$ and that it
coincides with
$p_{m+1}$ outside this set. Thus for all large $m$  we have
  \begin{equation*}
 I_{m+1}-I_m = \int_{S(P,e^{-m^{1/q}})}
      (p_{m+1}(x)-p_m(x))dx 
 \le \vol{S(P,e^{-m^{1/q}})}
=O\left(e^{-m^{1/q}/(2k)}\right)
  \end{equation*}
as $p_{m+1}(x)\le m+1$ and where we used (i), the implied constant is
independent of $m$. 
 Since $\sum_{m\ge 1} e^{-m^{1/q}/(2k)} <\infty$ we can use a
 telescoping sum to show that $\sup_{m\ge 0} I_m <\infty$, as
 desired. 
\end{proof}

Let $\IN_0$ denote the non-negative integers. Say
  $b\in\IN_0$  and let $g$ lie in  $C^b(\IR^n)$, the set of real valued
functions on $\IR^n$ 
whose derivatives exist up-to and including order $b$ and are continuous. 
For a multiindex $i=(i_1,\ldots,i_n)\in\IN_0^n$
we set $\ell(i) = i_1+\cdots+i_n$. If $\ell(i)\le b$, then we define
$\partial^i g = (\partial/\partial x_1)^{i_1}\cdots
(\partial/\partial x_n)^{i_n}g$ 
and 
\begin{equation*}
  |g|_{C^b} = \max_{\substack{i\in\IN_0^n \\  \ell(i)\le b}}
\sup_{x\in\IR^n} |\partial^i g(x)|
\end{equation*}
which is possibly $\infty$.

We now introduce a function that 
equals $\log |P(e(\cdot))|$ away from the singular locus but is
continuous on $\IR^n$ and attains $0$ when $P(e(\cdot))$ does. 

 Say
 $\phi \in C^b(\IR)$  is non-decreasing with
 $\phi(0)=0,\phi(1)=1$,
 $\phi(x) = 0$ if $x<0$, and $\phi(x)=1$ if $x>1$.
We ask in addition
that
$\partial^i \phi(0)=\partial^i(1)=0$
for all $i\in \{1,\ldots, b\}$. 
For example, we could take
 the anti-derivative of $x^b(1-x)^b$ that attains $0$  at
$x=0$,  scale it  to attain $1$ at $x=1$, and extend by $0$ for $x<0$
 and by $1$ for $x>1$. 

Say $\sv \in (0,1/2]$. We define $\phi_\sv$ as
$x\mapsto ((2/\sv)^2 x  -1)/3$, which rescales $[(\sv/2)^2,\sv^2]$ to
$[0,1]$, composed with $\phi$. Then
\begin{equation*}
 |\partial^i(\phi_\sv)(t)|\le  \left(\frac{4}{3\sv^2}\right)^i
 \left|\partial^i\phi(t)\right|
\end{equation*}
for all $i\in \{0,\ldots,b\}$ and all $t\in\IR$. Hence
\begin{equation}
\label{eq:phikbound}
 |\phi_\sv|_{C^b} =  O_{b,\phi}(\sv^{-2b}), 
\end{equation}
here and below the implied constant depends on the quantities
appearing in the subscript. 
Finally, we define $\psi_\sv$ as 
\begin{equation*}
  \psi_\sv(t)  = \left\{
  \begin{array}{ll}
    \frac 12 \phi_\sv(t) \log t &: t > 0, \\
    0 &: t \le 0.
  \end{array}\right.
\end{equation*}
Then  $\psi_\sv\in C^b(\IR)$ and all its derivatives up-to and including order $b$
vanish outside of $((\sv/2)^2,\sv^2)$.  
Thus
\begin{equation}
\label{eq:psiCbbound}
  |\psi_\sv|_{C^b} = O_{b,\phi}(\sv^{-2b}|\log\sv|)
\end{equation}
by the Leibniz product rule applied using (\ref{eq:phikbound}) and since $\sv \in (0,1/2]$. 

Say $P\in \IC[X_1,\ldots,X_n]$ 
and write $g(x) = |P(e(x))|^2$, this is a smooth function
$\IR^n\rightarrow\IR$ and
\begin{equation}
\label{eq:gCbbound}
  |g|_{C^b} =O_{b,P,n}(1). 
\end{equation}

We define $f_\sv = \psi_y\circ g$. If $P(e(x))\not=0$
with $x\in\IR^n$, then 
\begin{equation*}
f_\sv(x)=
 \psi_\sv(|P(e(x))|^2) = 
\phi_\sv(|P(e(x))|^2) \log |P(e(x))|
\end{equation*}
and for any $x\in\IR^n$ we have
\begin{equation*}
f_\sv(x) = \left\{
\begin{array}{ll}
  0 & :\text{ if $|P(e(x))|\le \sv/2$,}\\
  \log |P(e(x))| & :\text{ if $|P(e(x))|\ge\sv$.}
\end{array}\right. 
\end{equation*}
Moreover, the  composition $f_\sv :
\IR^n\rightarrow [0,1]$ lies in $C^b(\IR^n)$. We can bound its norm
using the following lemma. 

\begin{lemma}
  Let $b\ge 0, \psi \in C^b(\IR)$, and $g\in C^b(\IR^n)$ satisfy
  $|\psi|_{C^b} <\infty$ and $|g|_{C^b}<\infty$. Then
$|\psi\circ g|_{C^b} \le 2^{b(b-1)/2} |\psi|_{C^b} \max\{1,|g|_{C^b}\}^b$.
\end{lemma}
\begin{proof}
  The lemma is evident for $b=0$. So say $b\ge 1$ and let
 $i\in\IN_0^n\ssm\{0\}$ with $\ell=\ell(i)\le b$. 
Let $j\in\IN_0^n$ be a standard basis vector
with $i-j\in\IN_0^n$. 
Using the chain rule and the Leibniz product rule we
  find
  \begin{equation*}
    |\partial^i (\psi\circ g)|_{C^0} 
=    |\partial^{i-j} (\partial^j(\psi\circ g))|_{C^0} 
=    |\partial^{i-j} ((\psi'\circ g)\partial^j g)|_{C^0} 
\le 2^{\ell-1}
|\psi'\circ g|_{C^{\ell-1}}
|g|_{C^{\ell}}.
  \end{equation*}
  Thus
  $|\partial^i (\psi\circ g)|_{C^0}  \le 2^{b-1}
|\psi'\circ g|_{C^{b-1}}
|g|_{C^{b}} \le 2^{b(b-1)/2} |\psi|_{C^b}
\max\{1,|g|_{C^b}\}^{b}$, where
we applied this lemma by induction to bound
$|\psi'\circ g|_{C^{b-1}}$ from above.
This upper bound for 
$|\partial^i (\psi\circ g)|_{C^0}$ continues to hold for $i=0$ and it
is therefore an upper bound for $|\psi\circ g|_{C^b}$. 
\end{proof}

This lemma, together with (\ref{eq:psiCbbound}) and
(\ref{eq:gCbbound}), implies
\begin{equation}
\label{eq:fyCbbound}
  |f_\sv|_{C^b}  = O_{b,P,n,\phi}(\sv^{-2b}|\!\log\sv|).
\end{equation}

From now on we  suppose  $b\ge n+1$.
In the next three lemmas and if not stated otherwise,
$P\in\IC[X_1,\ldots,X_n]\ssm\{0\}$ is a polynomial with $k\ge 2$ non-zero
terms.

 \begin{lemma}
\label{lem:lawton1}
Suppose $a\in\IZ^n$ and $y\in(0,1/2]$, then
\begin{equation*}
\int_0^1 f_\sv(at) dt = \int_{[0,1)^n} f_\sv(x)
  dx+ O_{b,P,\phi,n}\left(
  \frac{|\!\log\sv|}{\sv^{2b}} \frac{1}{\rho(a;\IZ^n)^{b-n}}\right). 
\end{equation*}
 \end{lemma}
 \begin{proof}
All implied constants in this proof  depend only on $b,P,\phi,$ and $n$. 
The Fourier coefficients $\widehat{f_\sv}(m)$ of $f_\sv\in C^b(\IR)$,
 here $m\in\IZ^n$, 
 decay quickly. Indeed, by  Theorem
3.2.9(b) \cite{Grafakos}\footnote{$|\cdot|$ in the reference is the $\ell^2$-norm}
with $s=b$ and
$|\widehat{\partial^a f_y} (m)|\le |\partial^a f_y|_{C^0} \le
|\partial^a f_y|_{C^b}$ where $\ell(i)=b$.
We conclude
\begin{equation*}
  |\widehat{f_\sv}(m)| =O\left(
  \frac{|f_\sv|_{C^b}}{|m|^{b}}\right)\quad
\text{and so}\quad
  |\widehat{f_\sv}(m)|=O\left(\frac{|\!\log\sv|}{\sv^{2b} |m|^{b}}\right)
\end{equation*}
for all $m\in\IZ^n\ssm\{0\}$ by (\ref{eq:fyCbbound}). 
Say $H\ge 1$, then
$  \sum_{|m|\ge H} |\widehat{f_\sv}(m)| 
=O\left(\frac{|\!\log \sv|}{\sv^{2b}} 
\sum_{|m|\ge H} \frac{1}{|m|^{b}}\right)$, and hence
\begin{alignat}1
  \label{eq:fouriercoeffsumbound}
 \sum_{|m|\ge H} |\widehat{f_\sv}(m)| =O\left(
 \frac{|\!\log\sv|}{\sv^{2b}}\frac{1}{H^{b-n}}\right)
\end{alignat}
 as $b-n\ge 1$

   Since the Fourier coefficients of the continuous
function $f_\sv$ are absolutely
   summable, its Fourier series converges uniformly to $f_\sv$. Hence
   \begin{equation*}
     \int_0^1 f_\sv(at) dt = \sum_{m\in\IZ^n} \int_0^1
     \widehat{f_\sv}(m)e^{2\pi \sqrt{-1}\langle a,m\rangle t}dt
= \widehat{f_\sv}(0)+
\sum_{\substack{m\in\IZ^n\ssm\{0\} \\ \langle a,m\rangle=0}} 
     \widehat{f_\sv}(m).
   \end{equation*}
Now $\widehat{f_\sv}(0)$ equals $\int_{[0,1)^n} f_\sv(x)dx$,
  so the estimate in the assertion follows from
  (\ref{eq:fouriercoeffsumbound}) and
\begin{equation*}
\left|  \sum_{\substack{m\in\IZ^n\ssm\{0\} \\ \langle a,m\rangle=0}} 
     \widehat{f_\sv}(m)\right|\le
\sum_{|m|\ge \rho(a;\IZ^n)} 
|\widehat{f_\sv}(m)| =O\left(\frac{|\!\log\sv|}{\sv^{2b}} \frac{1}{\rho(a;\IZ^n)^{b-n}}\right).\qedhere
\end{equation*}
 \end{proof}

 \begin{lemma}
\label{lem:lawton2}
Suppose each non-zero term of $P$ has modulus at least $1$. 
   If $a\in\IZ^n$ such that
$\rho(a;\IZ^n)$ is sufficiently large in terms of $P$. Then
$|P(e(as))|$ is non-zero for all $s\in [0,1]$ 
outside a finite set of points. Moreover, if $\sv\in (0,1/2]$ then
   \begin{equation*}
 \int_0^1 \log|P(e(as))|ds = \int_0^1  f_\sv(as) ds + O\left(
     \sv^{1/(k-1)}|\!\log\sv|\right)
   \end{equation*}
where the implicit constant depends only on $P$ and $n$. 
 \end{lemma}
 \begin{proof}
Say $a=(a_1,\ldots,a_n)$ and suppose
that $\rho(a;\IZ^n)$ is strictly larger than $|m-m'|$ for all distinct
$m,m'$
   in the support of $P$. Then $P(X^{a_1},\ldots,X^{a_n})$ 
has the same coefficients as $P$. It is in particular non-zero and 
the first claim holds true. 
For the second claim we can apply Lawton's Lemma 4 \cite{Lawton}.
 \end{proof}

 \begin{lemma}
\label{lem:lawton3}
   Let $\epsilon > 0$. 
If $\sv\in (0,1/2]$ then
   \begin{equation*}
     \left| \int_{[0,1)^n} (f_\sv(x)-\log |P(e(x))|)dx\right|
=O\left( \sv^{\frac{1}{2(k-1)}-\epsilon}\right)
   \end{equation*}
where the implied constant depends only on $n,P,$ and $\epsilon$. 
 \end{lemma}
 \begin{proof}
Let $p>1$  and fix $q>1$ such that $1/p+1/q=1$. 
   By definition we get the equality in 
   \begin{alignat*}1
          &\left| \int_{[0,1)^n} (f_\sv(x)-\log |P(e(x))|)dx\right|
 = 
\left| \int_{[0,1)^n} (\phi_\sv(|P(e(x))|^2)-1)\log |P(e(x))|dx\right|\\
&\qquad\qquad\le 
\int_{[0,1)^n} |\phi_\sv(|P(e(x))|^2)-1|\left|\log |P(e(x))|\right|
dx \\ &\qquad\qquad\le
\left(\int_{[0,1)^n}|\phi_\sv(|P(e(x))|^2)-1|^pdx\right)^{1/p}
\left(\int_{[0,1)^n}\left|\log |P(e(x))|\right|^qdx\right)^{1/q}
   \end{alignat*}
we used the H\"older inequality in the last step; the final
integral is finite
by Lemma \ref{lem:volSPep}(ii).
As $\phi_\sv(|P(e(x))|^2) = 1$ if $|P(e(x))|\ge \sv$ we have
\begin{equation*}
\int_{[0,1)^n}|\phi_\sv(|P(e(x))|^2)-1|^p dx= 
\int_{S(P,\sv)}  |\phi_\sv(|P(e(x))|^2)-1|^p dx
\le\vol{S(P,\sv)}.
\end{equation*}
Recall $k\ge 2$. 
The current Lemma  follows from Lemma \ref{lem:volSPep}(i)
because we may suppose $1/(2p(k-1)) \ge 1/(2k-2) -\epsilon$.
 \end{proof}

 \begin{proof}[Proof of Proposition \ref{thm:lawtonquant}]
Without loss of generality, we may suppose that all non-zero
coefficients of $P$ have modulus at least $1$ and that $P$ is a
polynomial. We fix $b\ge n+1$ sufficiently large
and a sufficiently small
$\epsilon' > 0$ with 
\begin{equation}
\label{eq:chooseepprime}
-\frac{\gamma}{2(k-1)} + {\epsilon'\gamma} \le
-\frac{1}{4(k-1)}+\epsilon
\quad\text{where}\quad
\gamma = \frac{b-n}{2b + 1/(2k-2)}. 
\end{equation}
We fix  a step function
$\phi\in C^b(\IR)$ as above, abbreviate $H = \rho(a;\IZ^n)\ge 1$, and set $\sv = H^{-\gamma}$.  

For $H$ large enough
 we find $\sv \le 1/2$ and that $|\ma{P(X^{a_1},\ldots,X^{a_n})} - \ma{P}|$
equals
\begin{alignat*}1
&  \left|\int_{0}^1 \log |P(e(as))| ds - \int_{[0,1)^n}
\log |P(e(x))|dx \right| \\ &\qquad \le 
  \left|\int_{0}^1 f_\sv(as) ds - \int_{[0,1)^n}
f_\sv(x) dx \right| + 
  \left|\int_{0}^1 (\log |P(e(as))| - f_\sv(as)) ds\right| \\
&\qquad\qquad +
  \left|\int_{[0,1)^n} (f_\sv(x)-\log |P(e(x))| ) dx\right|.
\end{alignat*}
Then by Lemmas \ref{lem:lawton1},
\ref{lem:lawton2},  and \ref{lem:lawton3}, the final one applied to
a sufficiently small $\epsilon'$, this
 sum is in
\begin{equation*}
  O\left( \frac{|\!\log\sv|}{\sv^{2b}} \frac{1}{H^{b-n}}+
  \sv^{\frac{1}{2(k-1)}-\epsilon'}\right)
\end{equation*}
where the implied constant here and below depends only on  $b,P,\phi$,
and $\epsilon$. 
So the sum is in
\begin{equation*}
  O\left(\gamma (\log H)H^{2b \gamma +n-b}+
H^{-\frac{\gamma}{2(k-1)} + {\epsilon'\gamma}}\right)
\end{equation*}
The proposition follows 
for small enough $\epsilon'$ as $2b\gamma +n-b = -\gamma / (2k-2)$, cf. (\ref{eq:chooseepprime}).
 \end{proof}

\bibliographystyle{amsplain}
\bibliography{literature}

\def\cprime{$'$}
\providecommand{\bysame}{\leavevmode\hbox to3em{\hrulefill}\thinspace}
\providecommand{\MR}{\relax\ifhmode\unskip\space\fi MR }
\providecommand{\MRhref}[2]{%
  \href{http://www.ams.org/mathscinet-getitem?mr=#1}{#2}
}
\providecommand{\href}[2]{#2}
\begin{thebibliography}{10}

\bibitem{Autissier:counterexample}
P.~Autissier, \emph{Sur une question d'\'equir\'epartition de nombres
  alg\'ebriques}, C. R. Math. Acad. Sci. Paris \textbf{342} (2006), no.~9,
  639--641.

\bibitem{Ax:Schanuel}
J.~Ax, \emph{On {S}chanuel's conjectures}, Ann. of Math. (2) \textbf{93}
  (1971), 252--268.

\bibitem{BerndtEvansWilliams}
B.~C. Berndt, R.~J. Evans, and K.~S. Williams, \emph{Gauss and {J}acobi sums},
  Canadian Mathematical Society Series of Monographs and Advanced Texts, John
  Wiley \& Sons, Inc., New York, 1998.

\bibitem{BG}
E.~Bombieri and W.~Gubler, \emph{{H}eights in {D}iophantine {G}eometry},
  Cambridge University Press, 2006.

\bibitem{BMZGeometric}
E.~Bombieri, D.~W. Masser, and U.~Zannier, \emph{Anomalous {S}ubvarieties -
  {S}tructure {T}heorems and {A}pplications}, Int. Math. Res. Not. IMRN (2007),
  no.~19, 1--33.

\bibitem{Boyd:Kronecker}
D.~W. Boyd, \emph{Kronecker's theorem and {L}ehmer's problem for polynomials in
  several variables}, J. Number Theory \textbf{13} (1981), no.~1, 116--121.

\bibitem{Dimitrov}
V.~Dimitrov, \emph{{C}onvergence to the {M}ahler measure and the distribution
  of periodic points for algebraic {N}oetherian $\mathbb{Z}^d$-actions}, {\tt
  arXiv:1611.04664}.

\bibitem{Duke:Combinatorial}
W.~Duke, \emph{A combinatorial problem related to {M}ahler's measure}, Bull.
  Lond. Math. Soc. \textbf{39} (2007), no.~5, 741--748.

\bibitem{DukeGarciaLutz}
W.~Duke, S.~R. Garcia, and B.~Lutz, \emph{The graphic nature of {G}aussian
  periods}, Proc. Amer. Math. Soc. \textbf{143} (2015), no.~5, 1849--1863.

\bibitem{EinKapLin}
M.~Einsiedler, M.~Kapranov, and D.~Lind, \emph{Non-{A}rchimedean amoebas and
  tropical varieties}, J. Reine Angew. Math. \textbf{601} (2006), 139--157.

\bibitem{Grafakos}
L.~Grafakos, \emph{Classical {F}ourier analysis}, third ed., Graduate Texts in
  Mathematics, vol. 249, Springer, New York, 2014.

\bibitem{approxdefinable}
P.~Habegger, \emph{Diophantine {A}pproximations on {D}efinable {S}ets}, 2016
  preprint.

\bibitem{Lawton77}
W.~M. Lawton, \emph{A generalization of a theorem of kronecker}, J. Sci. Fat.
  Chiangmai Unit. (Thailand) \textbf{4} (1977), 15--23.

\bibitem{Lawton}
\bysame, \emph{A problem of {B}oyd concerning geometric means of polynomials},
  J. Number Theory \textbf{16} (1983), no.~3, 356--362.

\bibitem{LindSchmidtVerbitskiy}
D.~Lind, K.~Schmidt, and E.~Verbitskiy, \emph{Homoclinic points, atoral
  polynomials, and periodic points of algebraic {$\Bbb Z^d$}-actions}, Ergodic
  Theory Dynam. Systems \textbf{33} (2013), no.~4, 1060--1081.

\bibitem{Mahler:DiscIneq}
K.~Mahler, \emph{An inequality for the discriminant of a polynomial}, Michigan
  Math. J. \textbf{11} (1964), 257--262.

\bibitem{Mann}
H.~B. Mann, \emph{On linear relations between roots of unity}, Mathematika
  \textbf{12} (1965), 107--117.

\bibitem{Myerson:CombI}
G.~Myerson, \emph{A combinatorial problem in finite fields. {I}}, Pacific J.
  Math. \textbf{82} (1979), no.~1, 179--187.

\bibitem{Myerson:CombII}
\bysame, \emph{A combinatorial problem in finite fields. {II}}, Quart. J. Math.
  Oxford Ser. (2) \textbf{31} (1980), no.~122, 219--231.

\bibitem{Myerson:sumofrootsof1}
\bysame, \emph{Unsolved {P}roblems: {H}ow {S}mall {C}an a {S}um of {R}oots of
  {U}nity {B}e?}, Amer. Math. Monthly \textbf{93} (1986), no.~6, 457--459.

\bibitem{PilaWilkie}
J.~Pila and A.~J. Wilkie, \emph{The rational points of a definable set}, Duke
  Math. J. \textbf{133} (2006), no.~3, 591--616.

\bibitem{RahmanSchmeisser}
Q.~I. Rahman and G.~Schmeisser, \emph{Analytic theory of polynomials}, London
  Mathematical Society Monographs. New Series, vol.~26, The Clarendon Press,
  Oxford University Press, Oxford, 2002.

\bibitem{Schinzel}
A.~Schinzel, \emph{Polynomials with special regard to reducibility},
  Encyclopedia of Mathematics and its Applications, vol.~77, Cambridge
  University Press, Cambridge, 2000, With an appendix by Umberto Zannier.

\bibitem{SchmidtLNM1467}
W.M. Schmidt, \emph{Diophantine approximations and {D}iophantine equations},
  Lecture Notes in Mathematics, vol. 1467, Springer-Verlag, Berlin, 1991.

\bibitem{Smyth:Kronecker}
C.~J. Smyth, \emph{A {K}ronecker-type theorem for complex polynomials in
  several variables}, Canad. Math. Bull. \textbf{24} (1981), no.~4, 447--452.

\bibitem{SmythMahler}
\bysame, \emph{On measures of polynomials in several variables}, Bull. Austral.
  Math. Soc. \textbf{23} (1981), no.~1, 49--63.

\bibitem{D:oMin}
L.~van~den Dries, \emph{Tame topology and o-minimal structures}, London
  Mathematical Society Lecture Note Series, vol. 248, Cambridge University
  Press, Cambridge, 1998.

\bibitem{DM:96}
L.~van~den Dries and C.~Miller, \emph{Geometric categories and o-minimal
  structures}, Duke Math. J. \textbf{84} (1996), no.~2, 497--540.

\end{thebibliography}


\end{document}